\documentclass[12pt,reqno,a4paper]{amsart}
\usepackage{setspace}
\usepackage{fullpage}
\usepackage{datetime}

\usepackage{latexsym} 

\usepackage[utf8]{inputenc}
\usepackage[english]{babel}

\usepackage{amsfonts,amsmath,amssymb,amsbsy,amsthm,enumerate}

\usepackage[mathscr]{euscript}
\usepackage{pstricks}
\usepackage{multirow}
\usepackage{verbatim}
\usepackage{color}
\usepackage{comment} 
\usepackage[normalem]{ulem}
\usepackage{bm}

\definecolor{roxo}{HTML}{8613B0}

\newtheorem{theorem}{Theorem}[section]

\newtheorem{conjecture}[theorem]{Conjecture}

\newtheorem{lemma}[theorem]{Lemma}
\newtheorem{corollary}[theorem]{Corollary}

\newtheorem{claim}[theorem]{Claim}

\newtheorem{definition}[theorem]{Definition}

\newcommand{\aset}[1]{\left\{#1\right\}}

\usepackage{tikz}
\usepackage{xifthen}	
\usepackage{subcaption}
\usetikzlibrary{calc,positioning,decorations.pathmorphing,decorations.pathreplacing,trees,arrows,
positioning,scopes,shapes.gates.logic.US,fit}
\usepackage{./figstyle}

\DeclareMathOperator{\prehang}{preHang}
\DeclareMathOperator{\hang}{hang}

\def\back{\mathop{\text{\rm back}}\nolimits}
\def\forw{\mathop{\text{\rm forw}}\nolimits}

\newcommand{\D}{\mathcal{D}}

\newcommand{\B}{\mathcal{B}}

\newcommand{\full}{full}

\def\vertices{\mathop{\text{\rm vert}}\nolimits}

\newcommand{\margin}[1]{}
\newcommand{\marginl}[1]{}
\usepackage{graphicx}

\usepackage[colorlinks=false]{hyperref}

\title[Decomposing highly edge-connected graphs into paths of any given length]
{Decomposing highly edge-connected graphs into paths of any given length}

\author[all]{F. Botler, ~ G. O. Mota, ~  M. T. I. Oshiro, ~ Y. Wakabayashi}

\dedicatory{\small {\rm Instituto de Matem\'atica e Estat\'{\i}stica \\ Universidade de S\~ao 
Paulo, Brazil }}

\address{Instituto de Matem\'atica e Estat\'{\i}stica \\Universidade de 
   S\~ao Paulo, Rua do Mat\~ao 1010\\05508--090~S\~ao Paulo, Brazil
   {\rm(F. Botler, G. O. Mota, M. T. I. Oshiro, Y. Wakabayashi)}}
\email{\{fbotler|mota|oshiro|yw\}@ime.usp.br}

\thanks{This research has been partially supported by CNPq
  Projects (Proc. 477203/2012-4 and {456792/2014-7}), Fapesp Project
  (Proc. 2013/03447-6) and MaCLinC Project of Numec/USP, Brazil.
  F. Botler is supported by Fapesp (Proc. 2014/01460-8 and
  2011/08033-0), G. O.  Mota is supported by Fapesp
  (Proc. 2013/11431-2 and 2013/20733-2), M. T. I. Oshiro is supported
  by Capes, and Y. Wakabayashi is partially supported by CNPq Grant
  (Proc. 303987/2010-3).\\
 Email:{\texttt{{\{fbotler|mota|oshiro|yw\}@ime.usp.br}}}}

\date{\today, \currenttime}

\begin{document}

\begin{abstract}
  In 2006, Barát and Thomassen posed the following conjecture: for
  each tree~\(T\), there exists a natural number \(k_T\) such that, 
  if \(G\) is a \(k_T\)-edge-connected graph and \(|E(G)|\) is
  divisible by \(|E(T)|\), then \(G\) admits a decomposition into copies of~\(T\).  
  This conjecture was verified for stars, some bistars, paths of 
   length~$3$, $5$, and $2^r$ for every positive integer~$r$. 
   We prove that this conjecture holds for 
   paths of any fixed length.
\end{abstract}

\maketitle
\onehalfspacing

\section{Introduction}\label{section:intro}

A decomposition \(\D = \{H_1,\ldots,H_k\}\) of a graph \(G\) is a set 
of pairwise edge-disjoint subgraphs of \(G\) that cover the edge set of \(G\).
If \(H_i\) is isomorphic to a fixed graph \(H\) for \(1\leq i\leq k\), then 
we say that \(\D\) is an
\emph{\(H\)-decomposition} of \(G\).
It is known that, when \(H\) is connected and contains at least
\(3\)~edges, the  problem of deciding whether a graph admits an \(H\)-decomposition is
NP-complete~\cite{DorTarsi97}. 
When $H$ is a tree,
Barát and Thomassen~\cite{BaTh06} proposed the following conjecture,
that is the subject of our interest in this paper.

\begin{conjecture}\label{conj:dec}
  For each tree $T$, there exists a natural number $k_T$ such that, if $G$
  is a \hbox{$k_T$-edge-connected graph} and $|E(G)|$ is divisible by
  $|E(T)|$, then \(G\) admits a \hbox{\(T\)-decomposition.}
\end{conjecture}


The following version of Conjecture~\ref{conj:dec} for
bipartite graphs was shown by Barát and Gerbner~\cite{BaGe14}, and independently by
Thomassen~\cite{Th13a}, to be equivalent to Conjecture~\ref{conj:dec}.

\begin{conjecture}\label{conj:dec-bip}
  For each tree $T$, there exists a natural number $k'_T$ such that,
  if $G$ is a $k'_T$-edge-connected bipartite graph and $|E(G)|$ is
  divisible by $|E(T)|$, then \(G\) admits a \(T\)-decomposition.
\end{conjecture}

  Most of the known results on Conjecture~\ref{conj:dec} were
  obtained by Thomassen~\cite{Th12,Th08b,Th08a,Th13a,Th13b}: it holds
  for stars, paths of length~\(3\), a family of bistars, and for paths whose length is a power of~$2$.
  In 2014, we \cite{BoMoOsWa15-LAGOS} proved that it holds for paths of
    length~\(5\), and recently Merker~\cite{Me15+} proved that it holds for all trees
    with diameter at most~$4$, and also for some trees with diameter
    at most~$5$, including paths of length~5.

   In this paper we verify Conjecture~\ref{conj:dec-bip} (and
    Conjecture~\ref{conj:dec}) for paths of any given length.  More
    specifically, we prove that, for \(P_\ell\), the path of length $\ell$, 
    we have \(k_{P_{\ell}}' \leq 4\ell^2+10\ell-4\),
    if $\ell$ is odd; and \(k_{P_{\ell}}' \leq 26\ell + 8r-8\),
    with $r=\max\{32(\ell-1),\ell(\ell+2)\}$, if $\ell$ is even.

  In our proof (for $P_\ell$) we use a generalization of a 
  technique used by Thomassen~\cite{Th08b} to obtain an initial
  decomposition into trails of length~\(\ell\).
  We also borrow some ideas from a technique that we used
  in~\cite{BoMoWa14+} for regular graphs. A central part of this work
  concerns the ``disentangling'' of the undesired trails of our initial
  decomposition to construct a path decomposition.

  The paper is organized as follows.  In Section~\ref{sec:def} we give
  some definitions, establish the notation and state some auxiliary
  results needed in the proof of our main results. 
    In Section~\ref{section:disentangling} we present our main
    tool, called \emph{Disentangling Lemma}, that allows us to switch
    edges between the elements of a (special) trail decomposition so
    as to obtain a decomposition into paths.
  In Section~\ref{section:factorizations} we prove that highly
  edge-connected graphs admit well-structured decompositions with
  good properties that we can explore in the rest of the proof.  
  In Sections~\ref{section:odd-paths} and~\ref{section:even-paths} we
  present the results used in Section~\ref{section:high} to obtain the
  decompositions into paths of fixed odd and even length,
  respectively.  
   In Figure~\ref{fig:outline} we present a diagram that shows how
  the results (indicated in a rectangular box) are connected with each other,
  leading to the proof of our two main results,
  Theorems~\ref{thm:high-odd}~and~\ref{thm:high-even}. In this
  diagram, an arrow from a box A to a box B indicates that the result in A is
  used to prove the result in B.

  An extended abstract~\cite{BoMoOsWa15-euro1} of this work was
  accepted to \textsc{eurocomb 2015}. We have modified some previous
  terminology, but the techniques and results are essentially those we
  have mentioned in the extended abstract. This work grew out from our
  previous work on decomposition into paths of length
  five~\cite{BoMoOsWa14+}.  The reader may find useful to see the
  simpler ideas presented in this previous work, to get a better
  understanding of the technique used in this paper.


 \begin{figure}[h]
 	\centering
{

\tikzset{box/.style={
rectangle,
draw,
anchor=center,
text centered,
thick
},
}

\begin{tikzpicture}[scale = 0.7]

	\tikzstyle{every circle node} = [draw];
	\def\maxA{5}
	\def\maxB{8}
	\def\raioA{6}
	\def\raioB{2}

	\node (nash) [box,text width=2.8cm] at (-9.05,2.5) 
	      {Lemma \ref{lemma:splitting-of-vertices}\\Nash-Williams};
	\node (mader) [box,text width=2.8cm] at (-4.53,2.5) {Theorem \ref{thm:Mader}\\Mader};
	\node (no cut-vertex) [box,text width=2.8cm] at (0,2.5) {Lemma \ref{lemma:non-cut-vertex}\\
	      cut-vertex};
	\node (matchings) [box,text width=2.8cm] at (4.53,2.5) {Theorem \ref{thm:Kano}\\ 1-factor};
	\node (petersen) [box,text width=2.8cm] at (9.05,2.5) {Theorem \ref{thm:Petersen}\\ 
	      Petersen};
	
	\node (simple factorization) [box,text width=15.5cm] at (0,5) 
		{Lemma \ref{lemma:1-special-decomposition}\\Simple Fractional Factorization};
	
		
	\node (3-4) [box,text width=3.5cm] at (0,7.5) 
		{Theorem \ref{thm:edge-disjoint-spanning-trees}};

	\node (factorization) [box,text width=12.5cm] at (0,10) 
		{Lemma \ref{lem:tight-good-initial-decomposition}\\Initial Decomposition};	
		
	\node (4-4) [box,text width=2.2cm] at (-2,12.5) 
		{Lemma \ref{lemma:augmenting-sequence}};
	\node (4-7) [box,text width=2.2cm] at (2,12.5) 
		{Lemma \ref{lemma:choose-element-all}};	
		
	\node (disentangling) [box,text width=5cm] at (0,15) 
		{Lemma \ref{lemma:disentangling}\\Disentangling Lemma};	
	\node (3-12) [box,text width=1.6cm] at (9,15) 
		{Lemma \ref{lem:barat-part-3}};
		
	\node (6-2 and 6-3) [box,text width=2cm] at (-4.9,17.5) 
		{Lemmas\\\ref{lemma:3-PD} and \ref{lemma:odd-complete}};
	\node (6-4) [box,text width=1.8cm] at (-1.6,17.5) 
		{Lemma \ref{lemma:odd-path-dec}};
	\node (7-4) [box,text width=1.8cm] at (1.6,17.5) 
		{Lemma \ref{lemma:even-path-dec}};
	\node (7-2 and 7-3) [box,text width=2cm] at (4.9,17.5) 
		{Lemmas\\\ref{lemma:2-PD} and \ref{lemma:even-complete}};
	\node (5-6) [box,text width=2.2cm] at (9.4,17.5) 
		{Corollary\\\ref{corollary:even-special-decomposition}};
		
	\node (bifac odd) [box,text width=3cm] at (-8.85,20) 
		{Lemma \ref{lemma:odd-strong-factorization}\\Bifactorization};
	\node (bifac path odd) [box,text width=4cm] at (-3.3,20) 
		{Theorem \ref{thm:odd-path-decomposition}\\Bifact.$\Rightarrow$ Path dec.};
	\node (bifac path even) [box,text width=4cm] at (3.3,20) 
		{Theorem \ref{thm:even-path-decomposition}\\Bifact.$\Rightarrow$ Path dec.};
	\node (bifac even) [box,text width=3cm] at (8.85,20) 
		{Lemma \ref{lemma:even-strong-factorization}\\Bifactorization};
		
	\node (odd dec) [box,text width=7.3cm] at (-5.7,22.5) 
		{Theorem \ref{thm:high-odd}\\Odd-path decomposition};
	\node (even dec) [box,text width=7.3cm] at (5.7,22.5) 
		{Theorem \ref{thm:high-even}\\Even-path decomposition};

	

	\draw[->,thick] (nash) -- (-9.05,4.15);
	\draw[->,thick] (mader) -- (-4.53,4.15);
	\draw[->,thick] (no cut-vertex) -- (0,4.15);
	\draw[->,thick] (matchings) -- (4.53,4.15);
	\draw[->,thick] (petersen) -- (9.05,4.15);

	\draw[->,thick] (10.75,5.8) -- (10.75,16.71);
	\draw[->,thick] (-10.75,5.8) -- (-10.75,19.21);

	\draw[->,thick] (3-4) -- (0,9.21);

	\draw[->,thick] (7.22,10.8) -- (7.22,19.21);
	\draw[->,thick] (-7.22,10.8) -- (-7.22,19.21);
	
	\draw[->,thick] (4-4) -- (-2,14.15);
	\draw[->,thick] (4-7) -- (2,14.15);
	
	\draw[->,thick] (-1.68,15.8) -- (-1.68,16.7);
	\draw[->,thick] (1.68,15.8) -- (1.68,16.7);
	\draw[->,thick] (9,15.75) -- (9,16.7);

	\draw[->,thick] (-4.9,18.27) -- (-4.9,19.21);
	\draw[->,thick] (-1.6,18.27) -- (-1.6,19.21);
	\draw[->,thick] (1.6,18.27) -- (1.6,19.21);
	\draw[->,thick] (4.9,18.27) -- (4.9,19.21);
	\draw[->,thick] (9.4,18.27) -- (9.4,19.21);
	
	\draw[->,thick] (-8.85,20.75) -- (-8.85,21.65);
	\draw[->,thick] (-3.3,20.75) -- (-3.3,21.65);
	\draw[->,thick] (3.3,20.75) -- (3.3,21.65);
	\draw[->,thick] (8.85,20.75) -- (8.85,21.65);

%

\end{tikzpicture}
}

 	\caption{A diagram showing how the auxiliary results are
            used to build up the proof of the main results.}
 	\label{fig:outline}
 \end{figure}
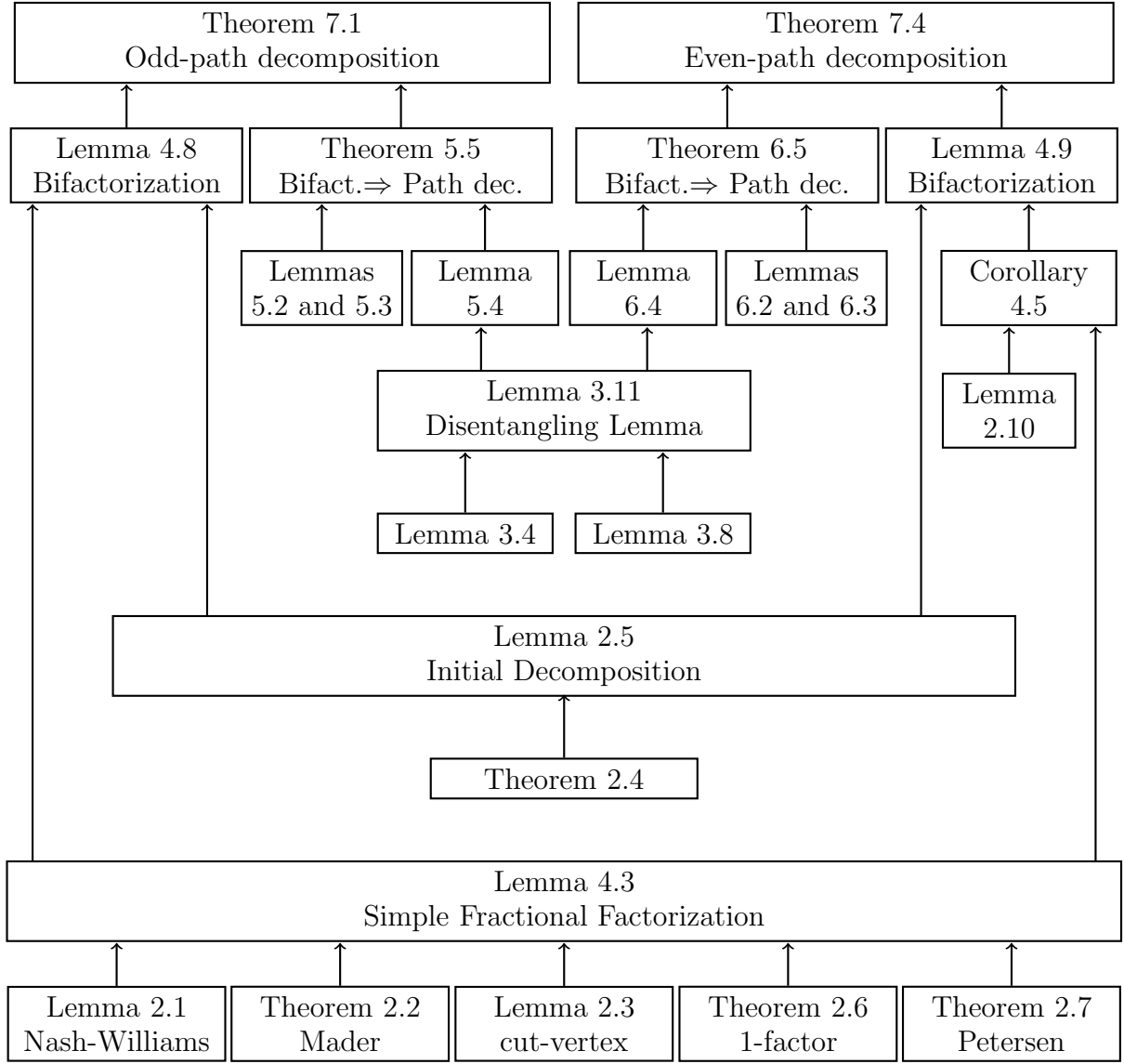


\section{Notation and auxiliary results}\label{sec:def}

The basic terminology and notation used in this paper are standard
(see, e.g.~\cite{Bo98,Di10}).  All graphs considered here are
finite and have no loops. 
Let $G=(V,E)$ be a graph.
A \emph{path}~$P$ in $G$ is a sequence of distinct vertices $P=v_0v_1\cdots v_\ell$ such
that $v_iv_{i+1}\in E$, for $0\leq i\leq \ell-1$. The
\emph{length} of a path $P$ is the number of its edges. The path of
length~\(\ell\),
also called \emph{$\ell$-path}, is denoted by $P_\ell$.  It is also
convenient to refer to a path \(P = v_0v_1\cdots v_\ell\)
as the subgraph of \(G\)
induced by the edges \(v_iv_{i+1}\) for \(i = 0,\ldots,\ell-1\).

 We denote by \(d_G(v)\) the 
\emph{degree} of a vertex \(v\in V\) and, when \(G\) is clear from the 
context, we write \(d(v)\). 
Given \(F\subset E\), we denote by $G[F]$ the subgraph of $G$ induced by $F$, and 
we also denote by \(d_{F}(v)\) the number of edges in \(F\) that are incident to \(v\).
An \emph{orientation} \(O\) of a subset \(F\subset E\),
is an assignment of a direction (from one of its vertices to the other) to each edge in \(F\).
If an edge \(e = uv\) in \(F\) is directed from \(u\) to \(v\),
we say that \(e\) \emph{leaves} \(u\) and \emph{enters} \(v\).
Given a vertex \(v\) of \(G\), we denote by \(d_O^+(v)\) (resp. \(d_O^-(v)\))
the number of the edges in \(F\) that leave (resp. enter) \(v\) in \(O\).
An \emph{Eulerian graph} is a graph that contains only vertices of even degree, and
an \emph{Eulerian orientation} of an Eulerian graph \(G\) 
is an orientation \(O\) of \(E\) such that \(d_O^+(v) = d_O^-(v)\) 
for every vertex \(v\) in \(V\). 
Note that an Eulerian graph does not need to be connected.
Furthermore, we say that a subset $F\subset E$ is \emph{Eulerian}
if \(G[F]\) is Eulerian.  We denote by $G=(A,B;E)$ a
bipartite graph $G$ on vertex classes $A$ and $B$.

We say that a set 
$\{H_1,\ldots,H_k\}$ of graphs is a
\emph{decomposition} of a graph $G$ if $\bigcup_{i=1}^k
E(H_i) = E$ and $E(H_i)\cap E(H_j)=\emptyset$ for all $1\leq i<j\leq
k$. Let~\(\mathcal{H}\) be a family of graphs. An
\emph{\(\mathcal{H}\)-decomposition} $\D$ of~\(G\) is a
decomposition of~\(G\) such that each element of $\D$ is
isomorphic to an element of~\(\mathcal{H}\). Furthermore, if~\(\mathcal{H} =
\{H\}\), then we say that $\D$ is an~\emph{\(H\)-decomposition}.


\subsection{Vertex splittings}
Let $G=(V,E)$ be a graph and $v$ a vertex of \(G\). 
A set $S_v=\{d_1,\ldots,d_{s_v}\}$ of $s_v$ positive integers is called a 
\emph{subdegree sequence for $v$} if
$d_1+\ldots+d_{s_v}=d_G(v)$. We say that a graph $G'$ is obtained 
by a \emph{$(v,S_v)$-splitting} 
of \(G\)
if $G'$ is composed of 
$G-v$ together with $s_v$ new vertices $v_1,\ldots,v_{s_v}$ and {$d_G(v)$ new edges} such that $d_{G'}(v_i)=d_i$, for $1\leq i\leq 
s_v$, and $\bigcup_{i=1}^{s_v} N_{G'}(v_i) = N_G(v)$.

Let $G$ be a graph and consider a set $V'=\{v_1,\ldots,v_r\}$ of $r$
vertices of $G$. Let $S_{v_1},\ldots,S_{v_r}$ be subdegree sequences
for $v_1,\ldots,v_r$, respectively. {Let $H_1,\ldots,H_r$ be graphs
  obtained as follows: $H_1$ is obtained by a
  $(v_1,S_{v_1})$-splitting of $G$, the graph $H_2$ is obtained by a
  $(v_2,S_{v_2})$-splitting of $H_1$, and so on, up to $H_r$, which is
  obtained by a $(v_r,S_{v_r})$-splitting of $H_{r-1}$.} We say that
each $H_i$ is an \emph{$\{S_{v_1},\ldots,S_{v_i}\}$-detachment} of
$G$. 
%
Roughly speaking, a
detachment of $G$ is a graph obtained by successive
applications of splitting operations on vertices of $G$. In
Figure~\ref{fig:detachment}, the graph $H$ is an
$\{S_a,S_e\}$-detachment of $G$, where $S_a=\{2,3\}$ and
$S_e=\{2,2,2\}$. The next result provides sufficient conditions for the existence of
$2k$-edge-connected detachments of $2k$-edge-connected graphs.

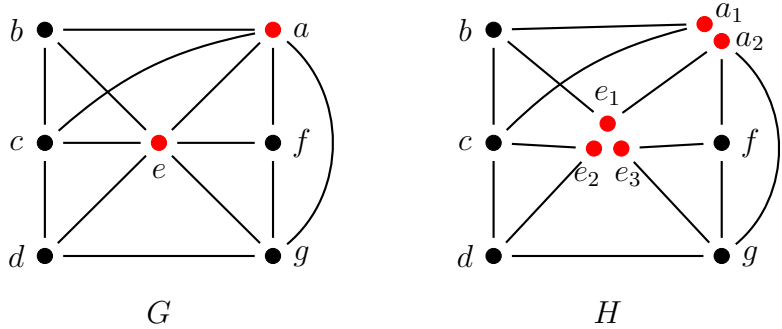
\begin{figure}[h]
	\centering
	\begin{tikzpicture}[scale = 1.5]


	
	\node (1) [black vertex] at (-.5,1.5) {};
	\node (2) [red vertex] at (1.5,1.5) {};
	
	\node (3) [black vertex] at (-.5,.5) {};
	\node (4) [red vertex] at (.5,.5) {};
	\node (5) [black vertex] at (1.5,.5) {};
	
	\node (6) [black vertex] at (-.5,-.5) {};
	\node (7) [black vertex] at (1.5,-.5) {};

	\node (b) [] at ($(1) - (0.25,0)$) {$b$};
	\node (a) [] at ($(2) + (0.25,0)$) {$a$};
	\node (c) [] at ($(3) - (0.25,0)$) {$c$};
	\node (e) [] at ($(4) - (0.0,+0.25)$) {$e$};
	\node (f) [] at ($(5) + (0.25,0)$) {$f$};
	\node (d) [] at ($(6) - (0.25,0)$) {$d$};
	\node (g) [] at ($(7) + (0.25,0)$) {$g$};
	
	\node (G) [] at ($(4) + (0.0,-1.5)$) {$G$};
	
	\draw[thick] (1) -- (2);
	\draw[thick] (1) -- (3);
	\draw[thick] (1) -- (4);
	\draw[thick] (2) -- (4);
	\draw[thick] (2) -- (5);
	\draw[thick] (3) -- (4);
	\draw[thick] (3) -- (6);
	\draw[thick] (4) -- (5);
	\draw[thick] (4) -- (6);
	\draw[thick] (4) -- (7);
	\draw[thick] (5) -- (7);
	\draw[thick] (6) -- (7);
	
	\draw[thick] (3) to [bend left=15] (2);
	\draw[thick] (7) to [bend right=50] (2);

\end{tikzpicture}\hspace{1cm}
	\begin{tikzpicture}[scale = 1.5]


	
	\node (1) [black vertex] at (-.5,1.5) {};
	
	\node (2a) [red vertex] at (1.35,1.55) {};
	\node (2b) [red vertex] at (1.5,1.4) {};
	
	\node (3) [black vertex] at (-.5,.5) {};
	
	\node (4a) [red vertex] at (.5,.67) {};
	\node (4b) [red vertex] at (.38,.45) {};
	\node (4c) [red vertex] at (.62,.45) {};
	
	\node (5) [black vertex] at (1.5,.5) {};
	
	\node (6) [black vertex] at (-.5,-.5) {};
	\node (7) [black vertex] at (1.5,-.5) {};

	\node (b) [] at ($(1) - (0.25,0)$) {$b$};
 	\node (a1) [] at ($(2a) + (0.22,0.1)$) {$a_1$};
 	\node (a2) [] at ($(2b) + (0.25,0)$) {$a_2$};
	\node (c) [] at ($(3) - (0.25,0)$) {$c$};
	\node (e1) [] at ($(4a) - (0.0,-0.25)$) {$e_1$};
	\node (e2) [] at ($(4b) - (0.06,+0.25)$) {$e_2$};
	\node (e3) [] at ($(4c) - (-0.06,+0.25)$) {$e_3$};
	\node (f) [] at ($(5) + (0.25,0)$) {$f$};
	\node (d) [] at ($(6) - (0.25,0)$) {$d$};
	\node (g) [] at ($(7) + (0.25,0)$) {$g$};
	
	\node (H) [] at ($(4a) + (0.0,-1.67)$) {$H$};
	
	\draw[thick] (1) -- (2a);
	\draw[thick] (1) -- (3);
	\draw[thick] (1) -- (4a);
	\draw[thick] (2b) -- (4a);
	\draw[thick] (2b) -- (5);
	\draw[thick] (3) -- (4b);
	\draw[thick] (3) -- (6);
	\draw[thick] (4c) -- (5);
	\draw[thick] (4b) -- (6);
	\draw[thick] (4c) -- (7);
	\draw[thick] (5) -- (7);
	\draw[thick] (6) -- (7);
	
	\draw[thick] (3) to [bend left=15] (2a);
	\draw[thick] (7) to [bend right=50] (2b);

\end{tikzpicture}
        \caption{{A graph $G$ and a graph $H$ that is an
          $\{S_a,S_e\}$-detachment of $G$.}}
	\label{fig:detachment}
\end{figure}

The next result provides
sufficient conditions for the existence of $2k$-edge-connected
detachments of $2k$-edge-connected graphs.

\begin{lemma}[Nash--Williams~\cite{Na85}]\label{lemma:splitting-of-vertices}
  {Let \(G\) be a \(2k\)-edge-connected graph, where $k\geq 1$,
    and 
    $V(G)=\{v_1,\ldots,v_n\}$.} For every $v\in V(G)$,  let
      $S_{v}=\{d^v_1,\ldots,d^v_{s_{v}}\}$ be a {subdegree sequence for
      $v$ such that $d^v_i\geq 2k$ for $i=1,\ldots,s_{v}$. Then, there
      exists a $2k$-edge-connected} 
      $\{S_{v_1},\ldots,S_{v_{n}}\}$-detachment of $G$.
\end{lemma}

\subsection{Edge liftings}
Let $G = (V,E)$ be a graph and $u,v,w$ be distinct vertices of $G$
such that $uv$, $vw \in E$. The multigraph
$G' = \big(V, (E \setminus \aset{uv,vw}) \cup \aset{uw}\big)$ is
called a \emph{\(uw\)-lifting}
(or, simply, a \emph{lifting}) at \(v\).
Note that $G'$ may have parallel edges connecting $u$ and $v$.
If for all distinct pairs $x$, $y \in V \setminus \aset{v}$, the
maximum number of edge-disjoint paths between $x$ and $y$ in $G'$ is
the same as in $G$, then the lifting at \(v\)
is called \emph{admissible}.  If \(v\)
is a vertex of degree \(2\),
then the lifting at \(v\)
is always admissible. Such a lifting together with the
deletion of \(v\)
is called a \emph{supression} of \(v\).
The next result is known as Mader's Lifting Theorem.

\begin{theorem}[Mader~\cite{Ma78}]\label{thm:Mader}
	Let $G$ be a  multigraph and $v$ a vertex of $G$. 
	If $v$ is not a cut-vertex, \hbox{$d_G(v) \geq 4$,} and $v$ has at least 2 neighbors, 
	then there exists an admissible lifting at $v$. 
\end{theorem}


The following simple lemma will be useful to apply Mader's
  Lifting Theorem.  In this lemma and thereafter, we denote by
  \(p_G(x,y)\) the maximum number of edge-disjoint paths between vertices \(x\)
  and \(y\) in a graph $G$.

\begin{lemma}\label{lemma:non-cut-vertex}
	Let \(G\) be a multigraph and let \(k\)
	be a positive integer. If $v$ is a vertex in~$G$ such that $d(v)<2k$ and
	$p_G(x,y)\geq k$ for any two distinct neighbors \(x\) and \(y\) of \(v\), 
	then $v$ is not a cut-vertex.
\end{lemma}


\subsection{High edge-connectivity}

If \(G\)
is a graph that contains \(2k\)
pairwise edge-disjoint spanning trees, then, clearly, \(G\)
is \hbox{\(2k\)-edge-connected.}
The converse is not true, but as stated in the next theorem,  every
  \(2k\)-edge-connected graph contains \(k\)
  such trees~\cite{Na61,Tu61}. 

 \begin{theorem}[Nash-Williams~\cite{Na61}; Tutte~\cite{Tu61}]
 \label{thm:edge-disjoint-spanning-trees}
 	Let \(k\) be a positive integer. 
 	If \(G\) is a \(2k\)-edge-connected graph, 
 	then \(G\) contains \(k\) pairwise edge-disjoint spanning trees.
 \end{theorem}

  Using Theorem~\ref{thm:edge-disjoint-spanning-trees} and a recent
  result of Lov{\'a}sz, Thomassen, Wu and Zhang~\cite{LoThWuZh13}, one
  can prove the following lemma, which enables us to treat highly
  edge-connected bipartite graphs as regular bipartite graphs. It is a
  slight generalization of Proposition~\(2\)
  in~\cite{Th13a}. A proof of this lemma is given
  in~\cite{BoMoOsWa14+}.

\begin{lemma}
\label{lem:tight-good-initial-decomposition}
	Let \(k\geq 2\) and \(r\) be positive integers.
	If \(G=(A_1, A_2;E)\) is a \((6k + 4r-4)\)-edge-connected bipartite 
	graph and \(|E|\) is divisible by 
	\(k\), 
	then \(G\) admits a decomposition into two spanning 
	\(r\)-edge-connected graphs \(G_1\) and \(G_2\) such 
	that,
	the degree in \(G_i\) of each vertex of \(A_i\) is divisible by \(k\),
	for $i=1,2$.
\end{lemma}

The following two results on regular multigraphs will be used
  later (see Figure~\ref{fig:outline}).


\begin{theorem}[Von Baebler~\cite{Ba37} (see also~{\cite[Theorem 
2.37]{AkKa11}})]\label{thm:Kano}
Let $r \geq 2$ be a positive integer, and $G$ be an $(r-1)$-edge-connected 
$r$-regular multigraph of even order. Then $G$ has a $1$-factor. 
\end{theorem}


\begin{theorem}[Petersen~\cite{Pe1891}]\label{thm:Petersen}
	If \(G\) is a \(2k\)-regular multigraph,
	then \(G\) admits a decomposition into \(2\)-factors.
\end{theorem}


The next results are obtained by generalizing a technique used
  by Bárat and Gerbner~\cite{BaGe14}.  They are useful in the proof of
  Lemma~\ref{lem:barat-part-3}, which is used to deal with
  decompositions into paths of even length.

\begin{theorem}[Theorem 20 in~\cite{ElNaVo02}]\label{thm:elnavo02}
  Let $m$ be a positive integer. If \(G\)
  is an \(m\)-edge-connected
  graph, then \(G\) contains a spanning tree \(T\) such that
\(d_T(v) \leq  \lceil d_G(v)/m\rceil + 2\) for every vertex \(v\).
\end{theorem}

\begin{corollary}\label{cor:razao}
   Let $m$ be a positive integer. If \(G\)
  is an \(m\)-edge-connected
  graph, then \(G\)
  contains a spanning tree \(T\)
  such that \(d_T(v) \leq 4\,d_G(v)/m\) for every vertex \(v\).
\end{corollary}

\begin{proof}
   From the edge-connectivity of $G$, we have
  \(d_G(v) \geq m\)
  for every vertex \(v\).
  Combining this with Theorem~\ref{thm:elnavo02}, we conclude
    that $G$ contains a spanning tree~$T$ such that
    $d_T(v) \leq \lceil d_G(v)/m\rceil + 2 \leq (d_G(v)/m)+3 \leq
    4\,d_G(v)/m$.
\end{proof}

\begin{lemma}\label{lem:barat-part-3}
Let $k$, $m$ and $r$ be positive integers,  and let $G=(A,B;E)$ be a 
bipartite graph. If $G$ is $8m\lceil(k+r)/k\rceil$-edge-connected and, for 
every 
$v\in A$, \(d_G(v)\) is divisible by \(k+r\), then 
\(G\) 
admits a decomposition into spanning graphs \(G_k\) and \(G_r\) such that \(G_k\) is 
\hbox{\(m\)-edge-connected} and, for every vertex  $v\in 
A$, we have \(d_{G_k}(v) = \frac{k}{k+r}d_G(v)\) and \(d_{G_r}(v) = 
\frac{r}{k+r}d_G(v)\).
\end{lemma}

\begin{proof}
Let $k$, $m$, $r$ and $G=(A,B;E)$ be as in the hypothesis of the
lemma. Since \(G\)
is \(8m\lceil(k+r)/k\rceil\)-edge-connected,
by Theorem~\ref{thm:edge-disjoint-spanning-trees}, we conclude that
\(G\)
contains at least \(4m\lceil(k+r)/k\rceil\)
pairwise edge-disjoint spanning trees.  Now, partition the set of 
  these
\(4m\lceil(k+r)/k\rceil\)
spanning trees into \(m\)
sets, say \(\mathcal{T}_1,\ldots,\mathcal{T}_m\),
of $4\lceil(k+r)/k\rceil$ spanning trees each,  and let \(G_i =
\bigcup_{T\in \mathcal{T}_i} T\), for \(i = 1,\ldots,m\).

Clearly, \(G_i\)
is $4\lceil(k+r)/k\rceil$-edge-connected.  By
Corollary~\ref{cor:razao}, \(G_i\)
contains a spanning tree \(T_i\) such that, for every $v\in V(G_i)$,
\[
	d_{T_i}(v) \leq \frac{1}{\lceil(k+r)/k\rceil} d_{G_i}(v) \leq 
\left(\frac{k}{k+r}\right)d_{G_i}(v).
\]
Let \(G' = \cup_{i=1}^m T_i\). Clearly, \(G'\) is \(m\)-edge-connected. Note 
that, for every $v\in V(G)$,
\begin{align*}
	d_{G'}(v) = \sum_{i=1}^m d_{T_i}(v) \leq \left(\frac{k}{k+r}\right) \sum_{i=1}^m 
d_{G_i}(v) \leq \left(\frac{k}{k+r}\right)d_G(v).
\end{align*}

Let $G_k$ be the bipartite graph obtained from $G'$ by adding,
for every vertex $v$ in $A$, exactly $\big((k/(k+r)\big)d_G(v) - d_{G'}(v)$
edges of $G-E(G')$ that are incident to $v$ (note that
$\big((k/(k+r)\big)d_G(v)$ is an integer). Therefore, every vertex
$v\in A$ has degree exactly $\big((k/(k+r)\big)d_G(v)$ in $G_k$. To
conclude the proof, take \(G_r = G - E(G_k)\).
\end{proof}


\section{The disentangling lemma}\label{section:disentangling}


Our aim in this section is to prove a result,
  Lemma~\ref{lemma:disentangling}, which guarantees that, given a
  special trail decomposition of a graph~$G$, it is possible to switch
  edges of the elements of this decomposition and construct a path
  decomposition of~$G$. For that, we introduce the concept of
  \emph{trackings} of a trail: they are important to specify the order
  in which the vertices of a trail are visited.
  
We came to know recently that the technique introduced in this section generalizes the one presented by Kouider and Lonc~\cite{KoLo99} for decompositions of girth-restricted even regular graphs into paths.
Here, we manage to overcome this girth condition, by requiring a sufficiently high minimum degree.

\subsection{Trails, trackings and augmenting sequences}

A \emph{trail} is a graph \(T\) for which there is a sequence
\(B=x_0\cdots x_\ell\) of its vertices (possibly with
  repetitions) such that \(E(T) = \{x_ix_{i+1}\colon 0\leq i\leq
\ell-1\}\); such a sequence is called a \emph{tracking} of \(T\), and
we say that \(T\) is the trail \emph{induced} by the tracking \(B\).
Note that a path admits only two possible trackings, while a cycle of
length \(\ell\) admits \(2\ell\) trackings.  The vertices \(x_0\) and
\(x_\ell\) are called \emph{end-vertices} of
\(B\).  

Given a tracking \(B=x_0\cdots x_\ell\), we denote by \(B^-\) the
tracking \(x_\ell \cdots x_0\), and, to ease notation, we denote
by \(V(B)\) and \(E(B)\) the sets \(\{x_0,\ldots,x_\ell\}\) of
vertices and \(\{x_ix_{i+1}\colon 0\leq i\leq \ell-1\}\) of edges of
\(B\), respectively. Moreover, we denote by \(\bar B\) the trail
\(\big(V(B),E(B)\big)\).

It will be convenient to say that a tracking \(B=x_0\cdots
  x_\ell\) \emph{traverses} the vertices $x_0,\ldots,x_\ell$ and the edges
  $x_0x_{1},\ldots,x_{\ell-1}x_{\ell}$ (in this order), and that
  $x_0x_1$ is the \emph{starting edge} of $B$
  and $x_{\ell-1}x_\ell$ is the \emph{ending edge} of $B$, or that $B$
  \emph{starts} with $x_0x_1$ and \emph{ends} with
  $x_{\ell-1}x_\ell$.
 
We say that a trail \(T\) is a \emph{vanilla trail}
if there is a tracking \(x_0x_1\cdots x_\ell\)
of \(T\) such that \(x_1\cdots x_{\ell-1}\) induces a path in
\(G\). A tracking that induces a vanilla trail is also called a
  \emph{vanilla tracking}.  (See Figure~\ref{fig:vanillas}.)


If a vanilla trail contains $\ell$ edges, then we say
that it is a \emph{vanilla $\ell$-trail}.  A set \(\B\) of pairwise
edge-disjoint trackings of vanilla \(\ell\)-trails of a graph
  $G$ is an \emph{\(\ell\)-tracking decomposition} of \(G\) if
\(\bigcup_{B\in \B} E(B) = E(G)\), i.e, \(\{\bar B\colon B\in \B\}\)
is a decomposition of \(G\) into vanilla \(\ell\)-trails.  If every
element of \(\B\) induces an \(\ell\)-path, then we say that \(\B\) is
an \emph{\(\ell\)-path tracking decomposition}.  We may omit the
  length~$\ell$, when it is clear from the context.
We note that if \(B_i\) and \(B_j\)
are trackings of a tracking decomposition \(\B\)
such that \(E(B_i)\cap E(B_j) \neq \emptyset\),
then \(\bar B_i = \bar B_j\) (that is, \(B_i\) and \(B_j\) induce the same vanilla trail).

\begin{figure}[h]
	\begin{subfigure}[h]{0.30\textwidth}
		\centering
		\begin{tikzpicture}[scale = .4]


\foreach \i in {0,...,3}
    {
	\node (\i) [] at (180 +\i*60:3) {};
    };

	\node (a) [black vertex] at (0) {};
	\node (b) [white vertex] at (1) {};
	\node (c) [white vertex] at (2) {};
	\node (d) [black vertex] at (3) {};
	
	\draw[E edge][decorate,decoration={snake,segment length=7mm}] 
				(a) to [bend right=60] (b);	
	\draw[E edge][decorate,decoration={snake,segment length=7mm}] 
				(b) to [bend right=15] (c);
	\draw[E edge][decorate,decoration={snake,segment length=7mm}] 
				(c) to [bend right=60] (d);
	\draw[E edge] (a) -- (b);
	\draw[E edge] (d) -- (c);
		
\node (x) [] at (0,1.5) {};		
\node (z) [] at (0,-3.5) {};
\end{tikzpicture}
		\caption{}
		\label{fig:vanilla-a}
	\end{subfigure}
	\begin{subfigure}[h]{0.30\textwidth}
		\centering
		\begin{tikzpicture}[scale = .4]


\foreach \i in {0,...,3}
    {
	\node (\i) [] at (180 +\i*60:3) {};
    };

	\node (a) [black vertex] at (0) {};
	\node (b) [white vertex] at (1) {};
	\node (c) [white vertex] at (2) {};
	\node (d) [black vertex] at (3) {};
	
	\draw[E edge][decorate,decoration={snake,segment length=7mm}] 
				(a) to [bend right=60] (b);	
	\draw[E edge][decorate,decoration={snake,segment length=7mm}] 
				(b) to [bend right=15] (c);
	\draw[E edge][decorate,decoration={snake,segment length=7mm}] 
				(c) to [bend right=60] (d);
	\draw[E edge] (a) -- (c);
	\draw[E edge] (d) -- (b);
		
\node (x) [] at (0,1.5) {};		
\node (z) [] at (0,-3.5) {};
\end{tikzpicture}
		\caption{}
		\label{fig:vanilla-b}
	\end{subfigure}
	\begin{subfigure}[h]{0.30\textwidth}
		\centering
		\begin{tikzpicture}[scale = .4]


\foreach \i in {0,...,2}
    {
	\node (\i) [] at (180 +\i*90:3) {};
    };

	\node (a) [black vertex] at (0) {};
	\node (b) [white vertex] at (1) {};
	\node (c) [black vertex] at (2) {};
	
	\draw[E edge][decorate,decoration={snake,segment length=7mm}] 
				(a) to [bend right=40] (b);	
	\draw[E edge][decorate,decoration={snake,segment length=7mm}] 
				(b) to [bend right=40] (c);
	\draw[E edge] (a) -- (b);
	\draw[E edge] (c) -- (b);
		
\node (x) [] at (0,1.5) {};		
\node (z) [] at (0,-3.5) {};
\end{tikzpicture}
		\caption{}
		\label{fig:vanilla-c}
	\end{subfigure}

\vspace{.5cm}
	\begin{subfigure}[h]{0.30\textwidth}
		\centering
		\begin{tikzpicture}[scale = .4]


\foreach \i in {0,...,2}
    {
	\node (\i) [] at ($(180 +\i*90:3) - (0,.5)$) {};
    };
 	\node (4) [] at ($(90:2) - (0,1)$) {};   

	\node (a) [black vertex] at (0) {};
	\node (c) [black vertex] at (2) {};
	
	\node (d) [white vertex] at (4) {};
	
	\draw[E edge][decorate,decoration={snake,segment length=7mm}] 
				(a) to [bend right=40] (c);	
	\draw[E edge] (a) -- (d);
	\draw[E edge] (c) -- (d);

\node (x) [] at (0,1.5) {};		
\node (z) [] at (0,-3.5) {};
\end{tikzpicture}
		\caption{}
		\label{fig:vanilla-d}
	\end{subfigure}
	\begin{subfigure}[h]{0.30\textwidth}
		\centering
		\begin{tikzpicture}[scale = .4]


\foreach \i in {0,...,2}
    {
	\node (\i) [] at ($(180 +\i*90:3) - (0,.5)$) {};
    };
 	\node (4) [] at ($(90:2) - (0,1)$) {};   

	\node (a) [black vertex] at (0) {};
	\node (b) [white vertex] at (1) {};
	\node (c) [black vertex] at (2) {};
	
	\node (d) [white vertex] at (4) {};
	
	\draw[E edge][decorate,decoration={snake,segment length=7mm}] 
				(a) to [bend right=40] (b);	
	\draw[E edge][decorate,decoration={snake,segment length=7mm}] 
				(b) to [bend right=40] (c);
					
	\draw[E edge] (a) -- (b);
	\draw[E edge] (c) -- (d);

\node (x) [] at (0,1.5) {};		
\node (z) [] at (0,-3.5) {};
\end{tikzpicture}
		\caption{}
		\label{fig:vanilla-e}
	\end{subfigure}
	\begin{subfigure}[h]{0.30\textwidth}
		\centering
		\begin{tikzpicture}[scale = .4]


\foreach \i in {0,...,2}
    {
	\node (\i) [] at ($(180 +\i*90:3) - (0,.5)$) {};
    };
 	\node (4) [] at ($(90:2) - (0,1)$) {};   

	\node (b) [white vertex] at (1) {};
	\node (c) [white vertex] at (2) {};
	
	\node (d) [black vertex] at (4) {};
	
	\draw[E edge][decorate,decoration={snake,segment length=7mm}] 
				(d) to [bend right=80] (b);	
	\draw[E edge][decorate,decoration={snake,segment length=7mm}] 
				(b) to [bend right=40] (c);
					
	\draw[E edge] (d) -- (b);
	\draw[E edge] (c) -- (d);

\node (x) [] at (0,1.5) {};		
\node (z) [] at (0,-3.5) {};
\end{tikzpicture}
		\caption{}
		\label{fig:vanilla-f}
	\end{subfigure}	
	
\caption{Examples of vanilla trails.}
\label{fig:vanillas}
\end{figure}
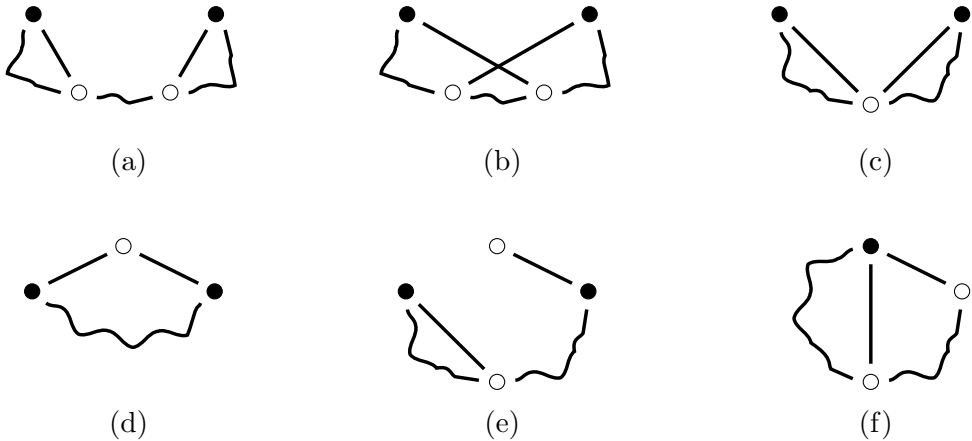




The concept of augmenting sequence (Definition~\ref{def:augmenting})
is central in this section. Before presenting it, we give a motivation for it.

For every vanilla trail \(T\) of \(G\), let \(\tau(T)\)
be the number of end-vertices of \(T\) with degree greater than \(1\).  
Let \(\mathcal{D}\)
be a decomposition of $G$ into vanilla \(\ell\)-trails
that minimizes \(\tau(\D)=\sum_{T\in\mathcal{D}}\tau(T)\).
If \(\tau(\D) = 0\), then \(\D\)
is an \(\ell\)-path decomposition.
So, let us assume that \(\tau(\D) > 0\). Moreover, 
suppose that \(\mathcal{D}\) has the following property: for every \(T\) in \(\mathcal{D}\)
and every vertex \(v\) of \(T\), there is a trail \(T'\)
containing an edge \(vu\), such that \(u\notin V(T)\)
and \(u\) is an end-vertex of \(T'\). 

Since \(\tau(\D)>0\), there is a vanilla trail \(T_0\)
in \(\mathcal{D}\) that is not a path.  Let \(x\)
be an end-vertex of \(T_0\) of degree greater than \(1\),
and let \(C\) be a cycle in \(T_0\)
that contains \(x\). Consider a neighbour \(v\)
of \(x\) in \(C\), and let \(T_1\)
be an element of \(\mathcal{D}\)
that contains an edge \(vu\),
such that \(u\notin V(T_0)\) and \(u\)
is an end-vertex of \(T_1\), as supposed above. 
 Now, let \(T'_0 = T_0 - vx + vu\), \(T'_1 = T_1 -vu + vx\),
and put \(\mathcal{D}' = \mathcal{D} - T_0 -T_1 + T'_0+T'_1\).
We have \(\tau(T'_0) = \tau(T_0) - 1\).
If \(\tau(T'_1) \leq \tau(T_1)\), then \(\mathcal{D}'\)
is a decomposition of \(G\) into vanilla \(\ell\)-trails such that
\(\sum_{T\in\mathcal{D}'}\tau(T) < \sum_{T\in\mathcal{D}}\tau(T)\),
a contradiction to the minimality of \(\tau(\D)\).
Therefore, \(\tau(T'_1) = \tau(T_1) + 1\) and \(T'_1\)
contains a cycle \(C'\) that contains \(xv\).
Now, we consider a neighbour \(v'\) of \(x\) in \(C'\)
such that \(v' \neq v\), and we repeat this operation as long as 
necessary considering \(T'_1\) and \(v'\)
instead of \(T_0\) and~\(v\).

We show that, under some assumptions, after repeating this
  operation a finite number of times, we obtain a better
  decomposition (an $\ell$-tracking decomposition in which there
  are more trackings inducing paths than the previous one).  
  The next definition formalizes which
  properties the sequence of trails must satisfy to guarantee this
  improvement. 

To formalize the ideas mentioned before, let us introduce some
concepts. Let \(\B\) be an \(\ell\)-tracking
decomposition of a graph \(G\), and let \(S = B_1B_2\cdots B_r\)
be a sequence of (not necessarily distinct) trackings of trails of
$G$, where \(B_i=b_0^ib_1^i\cdots b_\ell^i\), 
for $i=1\ldots,r$. We say that $S$ is a
\emph{\(\mathcal{B}\)-sequence} if  \(B_i\in \mathcal{B}\) or \(B_i^-\in \mathcal{B}\), for \(i=1,\ldots,r\).

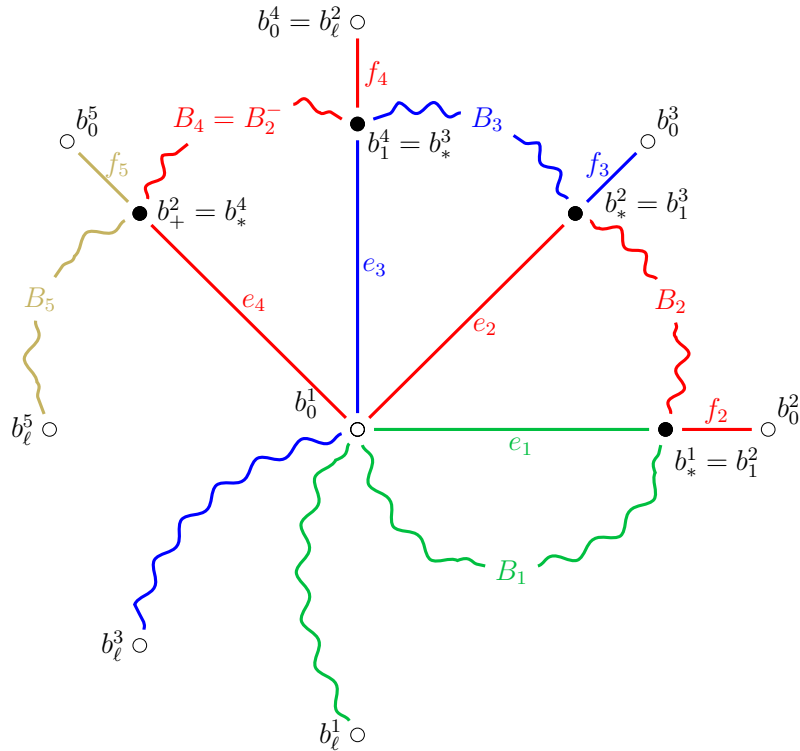
\begin{figure}[b] 
	\centering
	  \scalebox{.9}{\begin{tikzpicture}[scale = 1.5]

\node (c) [] at (0,0) {};

\foreach \i in {0,...,7}
    {
	\node (o\i) [] at (\i*45-45:3) {};
    };

\foreach \i in {0,...,7}
    {
	\node (om\i) [] at (\i*45-22.5:3.3) {};
    };

\foreach \i in {1,...,4}
    {
	\node (oo\i) [] at (\i*45-45:4) {};
    };

    
\foreach \i in {0,...,6}
    {
	\node (l\i) [] at (\i*45-6:1.6) {};
	\node (lf\i) [] at (\i*45+3:3.5) {};
	\node (lff\i) [] at (\i*45-3:3.5) {};	
    };

	\node (B1) [color=green!75!blue] at (1.5,-1.4) {\(B_1\)};
	\node (B2) [color=red] at (om1) {\(B_2\)};
	\node (B3) [color=blue] at (om2) {\(B_3\)};
	\node (B4) [color=red] at (om3) {\(B_4 = B_2^-\)};
	\node (B5) [color=yellow!75!blue] at (3.5*45:3.35) {\(B_5\)};

	\node (e1) [color=green!75!blue] at (l0) {\(e_1\)};
	\node (e2) [color=red] at (l1) {\(e_2\)};
	\node (e3) [color=blue] at (l2) {\(e_3\)};
	\node (e4) [color=red] at (l3) {\(e_4\)};

	\node (f2) [color=red] at (lf0) {\(f_2\)};
	\node (f3) [color=blue] at (lf1) {\(f_3\)};
	\node (f4) [color=red] at (lff2) {\(f_4\)};
	\node (f5) [color=yellow!75!blue] at (lff3) {\(f_5\)};

	\node (x0) [white vertex] at (o7) {};
	\node (v0) [left] at ($(x0) - (0.05,0)$) {\(b_\ell^1\)};
	\node (x1) [white vertex] at (c) {};
	\node (v1) [below] at ($(x1) + (-0.5,0.5)$) {\(b_0^1\)};
	\node (x4) [black vertex] at (o1) {};
	\node (v4) [below] at ($(x4) + (0.5,-0.08)$) {\(b_{*}^1 = b_1^2\)};
	
	\draw[E edge,color=green!75!blue][decorate,decoration={snake,segment length=7mm}] 
				(x0) to [bend left=30] (x1);
	\draw[E edge,color=green!75!blue][decorate,decoration={snake,segment length=7mm}]
				(x1) to [bend right=30] (B1); 
	\draw[E edge,color=green!75!blue][decorate,decoration={snake,segment length=7mm}]
				(B1) to [bend right=30] (x4);
	\draw[E edge][color=green!75!blue] (x4) -- (x1);

	\node (y0) [white vertex] at (oo1) {};
	\node (b01) at ($(y0) + (0.2,+0.2)$) {\(b_0^2\)};
	\node (y1) [black vertex] at (x4) {};
	\node (y3) [black vertex] at (o2) {};
	\node (b31) [above] at ($(y3) + (0.7,-.15)$) {\(b_{*}^2 = b_1^3\)};
	\node (y4) [white vertex] at (c) {};

	\node (y5) [black vertex] at (o4) {};
	\node (b51) [left] at ($(y5) + (1.15,0)$) {\(b_{+}^2 = b_{*}^4\)};
	\node (y6) [black vertex] at (o3) {};
	\node (y7) [white vertex] at (oo3) {};
	\node (b51) [left] at ($(y7) - (0.05,0)$) {\(b_0^4 = b_\ell^2\)};

	\draw[E edge][color=red] (y0) -- (y1);
	\draw[E edge,color=red,decorate,decoration={snake,segment length=5mm}] 
				(y1) to [bend right=15] (B2);
	\draw[E edge,color=red,decorate,decoration={snake,segment length=5mm}] 
				(B2) to [bend right=15] (y3);	
	
	\draw[E edge][color=red] (y3) -- (y4);
	\draw[E edge][color=red] (y4) -- (y5);
	\draw[E edge,color=red,decorate,decoration={snake,segment length=5mm}] 
				(y5) to [bend left=15] (B4);
	\draw[E edge,color=red,decorate,decoration={snake,segment length=5mm}] 
				(B4) to [bend left=15] (y6);

	\draw[E edge][color=red] (y6) -- (y7);
	
	\node (z0) [white vertex] at (oo2) {};
	\node (b02) at ($(z0) + (0.2,0.2)$) {\(b_0^3\)};
	\node (z1) [black vertex] at (o2) {};
	\node (z3) [black vertex] at (o3) {};
	\node (b32) [right] at ($(z3) + (0,-0.2)$) {\(b_1^4=b_{*}^3\)};
	\node (z4) [white vertex] at (c) {};
	\node (z5) [white vertex] at (o6) {};
	\node (b52) [left] at ($(z5) - (0.05,0)$) {\(b_\ell^3\)};
	
	\draw[E edge][color=blue] (z0) -- (z1);
	\draw[E edge,color=blue,decorate,decoration={snake,segment length=5mm}] 
				(z1) to [bend right=15] (B3);
	\draw[E edge,color=blue,decorate,decoration={snake,segment length=5mm}] 
				(B3) to [bend right=15] (z3);	
	
	\draw[E edge][color=blue] (z3) -- (z4);
	\draw[E edge,color=blue,decorate,decoration={snake,segment length=7mm}] 
				(z4) to [bend right=30] (z5);

	\node (w0) [white vertex] at (oo4) {};
	\node (b02) at ($(w0) + (0.2,0.2)$) {\(b_0^5\)};
	\node (w1) [black vertex] at (o4) {};
	\node (w5) [white vertex] at (o5) {};
	\node (b52) [left] at ($(w5) - (0.05,0)$) {\(b_\ell^5\)};
	
	\draw[E edge][color=yellow!75!blue] (w0) -- (w1);
	\draw[E edge,color=yellow!75!blue,decorate,decoration={snake,segment length=7mm}] 
				(w1) to [bend right=15] (B5);
	\draw[E edge,color=yellow!75!blue,decorate,decoration={snake,segment length=7mm}] 
				(B5) to [bend right=15] (w5);

\end{tikzpicture}} 
	\caption{An augmenting sequence \(S =
          B_1B_2B_3B_4B_5\), where \(B_4=B_2^-.\)}
	\label{fig:augmenting}
\end{figure}

In what follows, we shall be interested in such
\(\mathcal{B}\)-sequences
\(S = B_1B_2\cdots B_r\),
in which the vertex $b_0^1$, the first vertex of $B_1$, plays an
important role. We require that each element
$B_i$ of~\(S\), except the last one (the tracking $B_r$),  contains the vertex $b_0^1$. We
denote by $s(i)$ \emph{the smallest positive index such that
  $b_{s(i)}^i = b_0^1$}. As we need to refer frequently to the vertex
$b_{s(i)-1}^i$ (the vertex of $B_i$ that is traversed before 
$b_0^1$) and the vertex $b_{s(i)+1}^i$ (the vertex of $B_i$ that is
traversed after  $b_0^1$), for ease of notation, we also denote
them by $b_*^i$ and $b_+^i$ (that is, $b_*^i := b_{s(i)-1}^i$ and
$b_+^i := b_{s(i)+1}^i$), respectively. 
In this context, we also give names to two special edges
of each $B_i$; these are $e_i := b_*^ib_0^1$  (the
edge traversed by $B_i$ ``to enter''~$b_0^1$), and $f_i := b_0^ib_1^i$ 
(the starting edge of  $B_i$). See Figure~\ref{fig:augmenting}.




\newpage



\begin{definition}\label{def:augmenting}
Let \(\ell\) and $r\geq 2$ be positive integers 
and let \(\B\) be an \(\ell\)-tracking decomposition of a graph \(G\).
Let \(S = B_1B_2\cdots B_r\) be a \(\mathcal{B}\)-sequence,
where \(B_i=b_0^ib_1^i\cdots b_\ell^i\) for $i=1\ldots,r$. We say that 
\(S\) is an \emph{augmenting sequence} of $\B$ if 
\begin{itemize}
	\item[(i1)] \(\bar B_1\) is not a path, and $d_{\bar B_1}(b_0^1) >
          1$, $b_0^2\notin V(B_1)$;
        \item[(i2)] \(b_1^r = b_*^{r-1}\);

	\hspace*{-1.9cm}and for $i=2,\ldots, r-1$, the following
        holds:

	\item[(ii)] $B_i$ contains $b_0^1$  and \(b_1^{i} = b_*^{i-1}\);
	\item[(iii)] if \(\bar B_{i} \neq \bar B_{h}\)
	for every \(h< i\), then \(b_0^{i+1}\notin V(B_i)\);
	\item[(iv)] if \(\bar B_i = \bar B_{1}\), then
          $b_0^{i+1}\notin V\big(\bar B_1 - e_1 + f_2\big) = V(\bar B_i) \cup \{b_0^2\}$. 
         
	\item[(v)] if \(\bar B_{i} = \bar B_{h}\) for some $1<h<i$, 
	then \(b_0^{i+1}\notin V(\bar B_h - f_h + e_{h-1} - e_h+ f_{h+1}\big)\).
\end{itemize}
If, in addition, \(b_0^1 \notin V(B_r)\), then we say that \(S\) is a 
\emph{\full-augmenting sequence}. 
\end{definition}



We note that whenever we delete edges of a trail (as in (iv) and (v)), we also
remove the isolated vertices that may result after the edge deletions.
We also observe that  (iv) implies that if $\bar B_i= \bar B_1$
  then $B_{i+1} \ne B_i$ (this will be used  later).

The main idea behind our central result is that, given a
  certain tracking decomposition, if we can find a \full-augmenting
  sequence, then we can find a better decomposition.  Thus, the
  conditions stated in Definition~\ref{def:augmenting} have the purpose
  of allowing interchanging of edges of the elements of a
  \full-augmenting sequence. Such an interchange will be performed
  starting from the first element $B_1$ and then going from $B_i$ to
  $B_{i+1}$. If the elements are all disctint, then the simple
  interchange we have mentioned in the motivation suffices, as long
  the items (i)--(iii) are satisfied. But, as the trail corresponding
  to some trackings may repeat, we need the conditions stated in (iv)
  and (v). Note that item (iv) requires that if
  \(\bar B_i = \bar B_{1}\),
  then the initial vertex of $B_{i+1}$ does not belong to the trail
  corresponding to the tracking to which $B_1$ was transformed 
  (that is, the trail $\bar B_1 - e_1 + f_2$). 
 Item (v) requires that if \(\bar B_{i} = \bar B_{h}\) for some
 $1<h<i$, then the initial vertex of $B_{i+1}$ does not belong to the
 trail corresponding to the tracking to which $B_h$ was transformed.
 In this case, since $1<h<i$, the original tracking $B_h$ has suffered
 two transformations. (Suppose $r>2$.)  $B_2$ suffers a first transformation (because of
 $B_1$), but then, the transformed $B_2$ plays the role of the
 original $B_1$, and so it is again transformed because of
 $B_3$. Thus, the condition stated in item (v) reflects this double
 transformation suffered by $B_h$. To understand this idea, consider
 the augmenting sequence shown in Figure~\ref{fig:augmenting}, where $B_4=B_2^-$ (that
 is, $\bar B_4 = \bar B_2$), and see the step-by-step transformations
 shown in Figure~\ref{fig:augmenting-steps}.

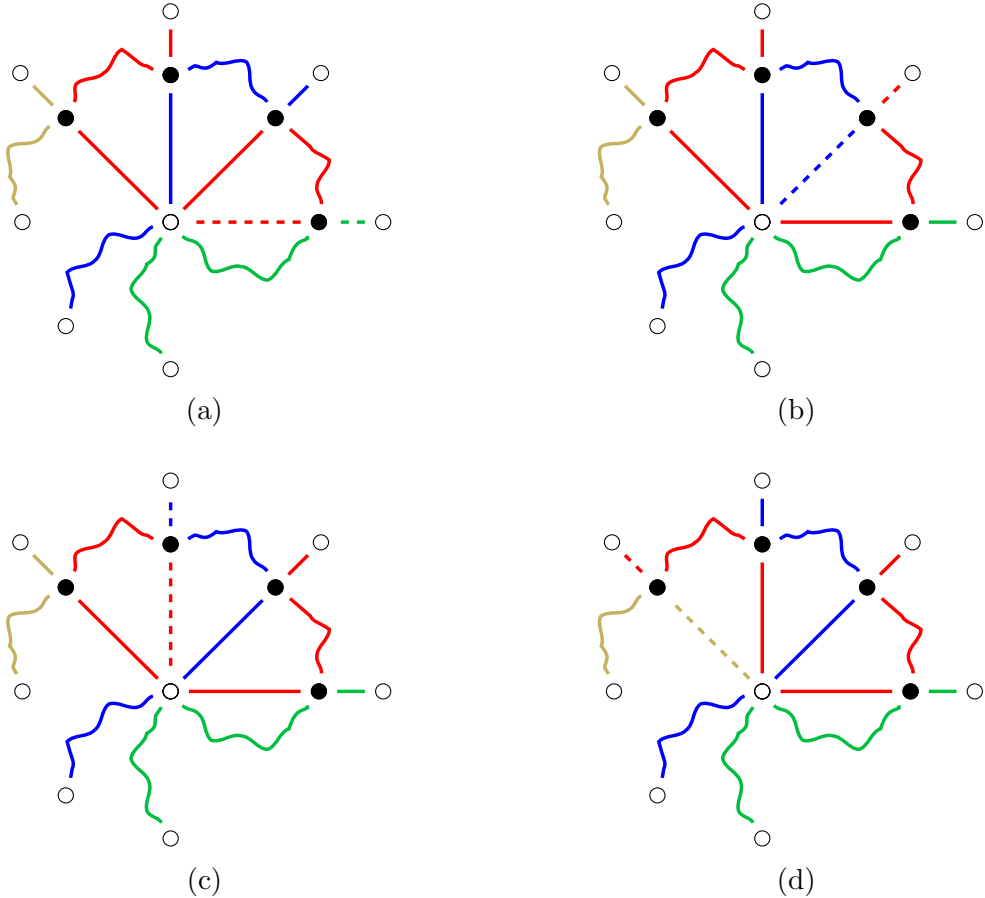
\begin{figure}[h]
	\centering
	\begin{subfigure}[h]{0.48\textwidth}
		\centering
		\begin{tikzpicture}[scale = .65]

\node (c) [] at (0,0) {};

\foreach \i in {0,...,7}
    {
	\node (o\i) [] at (\i*45-45:3) {};
    };
    
\foreach \i in {1,...,4}
    {
	\node (oo\i) [] at (\i*45-45:4.3) {};
    };

    
\foreach \i in {0,...,6}
    {
	\node (l\i) [] at (\i*45:1.7) {};
    };


	\node (x0) [white vertex] at (o7) {};
	\node (x1) [white vertex] at (c) {};
	\node (x4) [black vertex] at (o1) {};
	
	\draw[E edge,color=green!75!blue][decorate,decoration={snake,segment length=7mm}] 
				(x0) to [bend left=30] (x1);
	\draw[E edge,color=green!75!blue][decorate,decoration={snake,segment length=7mm}]
				(x1) to [bend right=50] (x4); 
	\draw[E edge][color=red,dashed] (x4) -- (x1);

	\node (y0) [white vertex] at (oo1) {};
	\node (y1) [black vertex] at (x4) {};
	\node (y3) [black vertex] at (o2) {};
	\node (y4) [white vertex] at (c) {};

	\node (y5) [black vertex] at (o4) {};
	\node (y6) [black vertex] at (o3) {};
	\node (y7) [white vertex] at (oo3) {};

	\draw[E edge][color=green!75!blue,dashed] (y0) -- (y1);
	\draw[E edge,color=red,decorate,decoration={snake,segment length=7mm}] 
				(y1) to [bend right=30] (y3);	
	\draw[E edge][color=red] (y3) -- (y4);
	\draw[E edge][color=red] (y4)  -- (y5);
	\draw[E edge,color=red,decorate,decoration={snake,segment length=7mm}] 
				(y5) to [bend left=40] (y6);
	\draw[E edge][color=red] (y6) -- (y7);
	
	\node (z0) [white vertex] at (oo2) {};
	\node (z1) [black vertex] at (o2) {};
	\node (z3) [black vertex] at (o3) {};
	\node (z4) [white vertex] at (c) {};
	\node (z5) [white vertex] at (o6) {};
	
	\draw[E edge][color=blue] (z0) -- (z1);
	\draw[E edge,color=blue,decorate,decoration={snake,segment length=7mm}] 
				(z1) to [bend right=40] (z3);	
	\draw[E edge][color=blue] (z3)  -- (z4);
	\draw[E edge,color=blue,decorate,decoration={snake,segment length=7mm}] 
				(z4) to [bend right=30] (z5);

	\node (w0) [white vertex] at (oo4) {};
	\node (w1) [black vertex] at (o4) {};
	\node (w5) [white vertex] at (o5) {};
	
	\draw[E edge][color=yellow!75!blue] (w0) -- (w1);
	\draw[E edge,color=yellow!75!blue,decorate,decoration={snake,segment length=7mm}] 
				(w1) to [bend right=40] (w5);	

\end{tikzpicture}
		\caption{}
		\label{fig:augmenting-a}
	\end{subfigure}
	\begin{subfigure}[h]{0.48\textwidth}
		\centering
		\begin{tikzpicture}[scale = .65]

\node (c) [] at (0,0) {};

\foreach \i in {0,...,7}
    {
	\node (o\i) [] at (\i*45-45:3) {};
    };
    
\foreach \i in {1,...,4}
    {
	\node (oo\i) [] at (\i*45-45:4.3) {};
    };

    
\foreach \i in {0,...,6}
    {
	\node (l\i) [] at (\i*45:1.7) {};
    };


	\node (x0) [white vertex] at (o7) {};
	\node (x1) [white vertex] at (c) {};
	\node (x4) [black vertex] at (o1) {};
	
	\draw[E edge,color=green!75!blue][decorate,decoration={snake,segment length=7mm}] 
				(x0) to [bend left=30] (x1);
	\draw[E edge,color=green!75!blue][decorate,decoration={snake,segment length=7mm}]
				(x1) to [bend right=50] (x4); 
	\draw[E edge][color=red] (x4) -- (x1);

	\node (y0) [white vertex] at (oo1) {};
	\node (y1) [black vertex] at (x4) {};
	\node (y3) [black vertex] at (o2) {};
	\node (y4) [white vertex] at (c) {};

	\node (y5) [black vertex] at (o4) {};
	\node (y6) [black vertex] at (o3) {};
	\node (y7) [white vertex] at (oo3) {};

	\draw[E edge][color=green!75!blue] (y0) -- (y1);
	\draw[E edge,color=red,decorate,decoration={snake,segment length=7mm}] 
				(y1) to [bend right=30] (y3);	
	\draw[E edge][color=blue, dashed] (y3) -- (y4);
	\draw[E edge][color=red] (y4)  -- (y5);
	\draw[E edge,color=red,decorate,decoration={snake,segment length=7mm}] 
				(y5) to [bend left=40] (y6);
	\draw[E edge][color=red] (y6) -- (y7);
	
	\node (z0) [white vertex] at (oo2) {};
	\node (z1) [black vertex] at (o2) {};
	\node (z3) [black vertex] at (o3) {};
	\node (z4) [white vertex] at (c) {};
	\node (z5) [white vertex] at (o6) {};
	
	\draw[E edge][color=red, dashed] (z0) -- (z1);
	\draw[E edge,color=blue,decorate,decoration={snake,segment length=7mm}] 
				(z1) to [bend right=40] (z3);	
	\draw[E edge][color=blue] (z3)  -- (z4);
	\draw[E edge,color=blue,decorate,decoration={snake,segment length=7mm}] 
				(z4) to [bend right=30] (z5);

	\node (w0) [white vertex] at (oo4) {};
	\node (w1) [black vertex] at (o4) {};
	\node (w5) [white vertex] at (o5) {};
	
	\draw[E edge][color=yellow!75!blue] (w0) -- (w1);
	\draw[E edge,color=yellow!75!blue,decorate,decoration={snake,segment length=7mm}] 
				(w1) to [bend right=40] (w5);	

\end{tikzpicture}
		\caption{}
	\end{subfigure}

\vspace{.5cm}	

	\begin{subfigure}[h]{0.48\textwidth}
		\centering
		\begin{tikzpicture}[scale = .65]

\node (c) [] at (0,0) {};

\foreach \i in {0,...,7}
    {
	\node (o\i) [] at (\i*45-45:3) {};
    };
    
\foreach \i in {1,...,4}
    {
	\node (oo\i) [] at (\i*45-45:4.3) {};
    };

    
\foreach \i in {0,...,6}
    {
	\node (l\i) [] at (\i*45:1.7) {};
    };


	\node (x0) [white vertex] at (o7) {};
	\node (x1) [white vertex] at (c) {};
	\node (x4) [black vertex] at (o1) {};
	
	\draw[E edge,color=green!75!blue][decorate,decoration={snake,segment length=7mm}] 
				(x0) to [bend left=30] (x1);
	\draw[E edge,color=green!75!blue][decorate,decoration={snake,segment length=7mm}]
				(x1) to [bend right=50] (x4); 
	\draw[E edge][color=red] (x4) -- (x1);

	\node (y0) [white vertex] at (oo1) {};
	\node (y1) [black vertex] at (x4) {};
	\node (y3) [black vertex] at (o2) {};
	\node (y4) [white vertex] at (c) {};

	\node (y5) [black vertex] at (o4) {};
	\node (y6) [black vertex] at (o3) {};
	\node (y7) [white vertex] at (oo3) {};

	\draw[E edge][color=green!75!blue] (y0) -- (y1);
	\draw[E edge,color=red,decorate,decoration={snake,segment length=7mm}] 
				(y1) to [bend right=30] (y3);	
	\draw[E edge][color=blue] (y3) -- (y4);
	\draw[E edge][color=red] (y4)  -- (y5);
	\draw[E edge,color=red,decorate,decoration={snake,segment length=7mm}] 
				(y5) to [bend left=40] (y6);
	\draw[E edge][color=blue,dashed] (y6) -- (y7);
	
	\node (z0) [white vertex] at (oo2) {};
	\node (z1) [black vertex] at (o2) {};
	\node (z3) [black vertex] at (o3) {};
	\node (z4) [white vertex] at (c) {};
	\node (z5) [white vertex] at (o6) {};
	
	\draw[E edge][color=red] (z0) -- (z1);
	\draw[E edge,color=blue,decorate,decoration={snake,segment length=7mm}] 
				(z1) to [bend right=40] (z3);	
	\draw[E edge][color=red,dashed] (z3)  -- (z4);
	\draw[E edge,color=blue,decorate,decoration={snake,segment length=7mm}] 
				(z4) to [bend right=30] (z5);

	\node (w0) [white vertex] at (oo4) {};
	\node (w1) [black vertex] at (o4) {};
	\node (w5) [white vertex] at (o5) {};
	
	\draw[E edge][color=yellow!75!blue] (w0) -- (w1);
	\draw[E edge,color=yellow!75!blue,decorate,decoration={snake,segment length=7mm}] 
				(w1) to [bend right=40] (w5);	

\end{tikzpicture}
		\caption{}
	\end{subfigure}
	\begin{subfigure}[h]{0.48\textwidth}
		\centering
		\begin{tikzpicture}[scale = .65]

\node (c) [] at (0,0) {};

\foreach \i in {0,...,7}
    {
	\node (o\i) [] at (\i*45-45:3) {};
    };
    
\foreach \i in {1,...,4}
    {
	\node (oo\i) [] at (\i*45-45:4.3) {};
    };

    
\foreach \i in {0,...,6}
    {
	\node (l\i) [] at (\i*45:1.7) {};
    };


	\node (x0) [white vertex] at (o7) {};
	\node (x1) [white vertex] at (c) {};
	\node (x4) [black vertex] at (o1) {};
	
	\draw[E edge,color=green!75!blue][decorate,decoration={snake,segment length=7mm}] 
				(x0) to [bend left=30] (x1);
	\draw[E edge,color=green!75!blue][decorate,decoration={snake,segment length=7mm}]
				(x1) to [bend right=50] (x4); 
	\draw[E edge][color=red] (x4) -- (x1);

	\node (y0) [white vertex] at (oo1) {};
	\node (y1) [black vertex] at (x4) {};
	\node (y3) [black vertex] at (o2) {};
	\node (y4) [white vertex] at (c) {};

	\node (y5) [black vertex] at (o4) {};
	\node (y6) [black vertex] at (o3) {};
	\node (y7) [white vertex] at (oo3) {};

	\draw[E edge][color=green!75!blue] (y0) -- (y1);
	\draw[E edge,color=red,decorate,decoration={snake,segment length=7mm}] 
				(y1) to [bend right=30] (y3);	
	\draw[E edge][color=blue] (y3) -- (y4);
	\draw[E edge][color=yellow!75!blue,dashed] (y4)  -- (y5);
	\draw[E edge,color=red,decorate,decoration={snake,segment length=7mm}] 
				(y5) to [bend left=40] (y6);
	\draw[E edge][color=blue] (y6) -- (y7);
	
	\node (z0) [white vertex] at (oo2) {};
	\node (z1) [black vertex] at (o2) {};
	\node (z3) [black vertex] at (o3) {};
	\node (z4) [white vertex] at (c) {};
	\node (z5) [white vertex] at (o6) {};
	
	\draw[E edge][color=red] (z0) -- (z1);
	\draw[E edge,color=blue,decorate,decoration={snake,segment length=7mm}] 
				(z1) to [bend right=40] (z3);	
	\draw[E edge][color=red] (z3)  -- (z4);
	\draw[E edge,color=blue,decorate,decoration={snake,segment length=7mm}] 
				(z4) to [bend right=30] (z5);

	\node (w0) [white vertex] at (oo4) {};
	\node (w1) [black vertex] at (o4) {};
	\node (w5) [white vertex] at (o5) {};
	
	\draw[E edge][color=red,dashed] (w0) -- (w1);
	\draw[E edge,color=yellow!75!blue,decorate,decoration={snake,segment length=7mm}] 
				(w1) to [bend right=40] (w5);	

\end{tikzpicture}
		\caption{}
	\end{subfigure}	
	
       \caption{Illustration of how to deal with the \full-augmenting sequence in
                        Figure~\ref{fig:augmenting}. In each step, the dashed edges are those 
              that are swichted.  }
      	\label{fig:augmenting-steps}

\end{figure}

%
As we will see,  full-augmenting sequences of a tracking
decomposition \(\B\) have a finite number of elements.  
To prove this
(Corollary~\ref{cor:augmenting-sequence-repeating-element}), 
we show first the following result.


\begin{lemma}\label{lemma:non-repeating}
Let \(\ell\) and $r\geq 2$ be positive integers 
and let \(\B\) be an \(\ell\)-tracking decomposition of a graph \(G\).
If \(S = B_1B_2\cdots B_r\) is an augmenting sequence of $\B$, where 
\(B_i=b_0^ib_1^i\cdots b_\ell^i\) for 
$i=1,\ldots,r$, then \(b^i_* \neq b^j_*\) for every $i$, $j$ with \(1\leq i<j\leq r-1\).
\end{lemma}

\begin{proof}
  Let \(\ell\), $r$, $\B$ and \(S = B_1B_2\cdots B_r\)
  be as in the hypothesis of the lemma.  We want to prove that
  \(\{b^1_{*},b^2_{*}, \ldots, b^{r-1}_{*}\}\) is a set of distinct elements.

\smallskip
	
\noindent\emph{Claim A:} $b^{j}_{*}\neq b^1_{*}$, for  $j=2,\ldots,r-1$.

For $j=2$ the result is immediate. Indeed, recall that
$e_i= b_{*}^ib_0^1 \in E(B_i)$.  If $b_{*}^2 = b_{*}^1$ then
$e_1=e_2$, that is, $\bar B_1$ and $\bar B_2$ have a common edge.  But
then, $\bar B_1 = \bar B_2$, a contradiction (to Definition~\ref{def:augmenting} (i1)). 
Now suppose $j\geq 3$. Take such a smallest index~$j$ for which
$b^{j}_{*} = b^1_{*}$. As in the previous case, we conclude that 
$\bar B_j = \bar B_1$. Since $S$ is a \(\B\)-sequence, 
either \(B_{j} = B_1^-\)  or $B_j = B_1$.
If \(B_{j} = B_1^-\), then $b_{*}^j = b_{+}^1$. But $b_{+}^1\neq
b_{*}^1$.  Thus, $b_{*}^j \neq b_{*}^1$, a contradiction. 
If \(B_{j} = B_1\), then  $b_0^j =b_0^1$  and
$f_j = f_1$.  But $b_1^j = b_{*}^{j-1}$ (by Definition~\ref{def:augmenting}
(ii)). Hence, $f_j = b_0^jb_1^j=b_0^1b_{*}^{j-1} = e_{j-1}$, that is, $f_j \in 
B_{j-1}$. Since $f_j=f_1 \in B_1$, we conclude that $\bar B_{j-1} =
\bar{B_1}$. Thus,  $\bar B_{j-1} = \bar{B_1}=\bar B_{j}$, that is, $\bar
B_{j-1} = \bar{B_j}$, a contradiction (see the observation after
Definition~\ref{def:augmenting}). 

\smallskip

\noindent\emph{Claim B:}  $b^{i}_{*}\neq b^{j}_{*}$ for every $i$, $j$
with $2\leq i< j\leq r-1$. 

Suppose that this does not hold. Let~$i$ be the
smallest integer such that there exists $j>i$ such that
$b^{i}_{*} = b^{j}_{*}$. 
In this case, $e_i=b_{*}^ib_0^1 = b_{*}^jb_0^1 = e_j$,
and thus $\bar B_i = \bar B_j$. Hence, either $B_j = B_i$ or $B_j =
B_i^-$. If $B_j=B_i^-$, then $b_{*}^j = b_{+}^i \neq b_{*}^i$, a
contradiction. If  $B_j=B_i$, then $b_0^j = b_0^i$. Hence, 
$f_j = b_0^jb_1^j = b_0^ib_{*}^{j-1}$ and 
$f_i = b_0^jb_1^i = b_0^ib_{*}^{i-1}$. Since $f_j=f_i$, we
conclude that $b_{*}^{j-1} = b_{*}^{i-1}$, a contradiction to the
choice of $i$. 
\end{proof}


\begin{corollary}\label{cor:augmenting-sequence-repeating-element}
  Let \(\ell\)
  and $r\geq 2$ be positive integers and let \(\B\)
  be an \(\ell\)-tracking
  decomposition of a graph \(G\).
  If \(S = B_1B_2\cdots B_r\)
  is an augmenting sequence of $\B$, then each \(B_i\)
  occurs at most once in \(S\).
  Furthermore, if $\bar B_i = \bar B_j$ for some pair $i$, $j$ with
  \(1\leq i<j \leq r-1\), then \(B_j = B_i^-\).
\end{corollary}

\begin{proof}
Let \(\ell\), $r$, \(\B\) and \(S = B_1B_2\cdots B_r\)  be as in the
hypothesis of the corollary. Let $B_i=b^i_0 b^i_1\cdots b^i_\ell$, for $i=1,\ldots,r$.
Suppose, for a contradiction,  that \(B_i = B_j\) for some pair $i$,
$j$ with  \(1\leq i<j\leq r-1\). In this case, $e_i=e_j$, and
therefore, \(b^i_{*} = b^j_{*}\), a contradiction to Lemma~\ref{lemma:non-repeating}.
Now, since \(S\) is a \(\B\)-sequence
and \(B_i\neq B_j\), if \(\bar B_i = \bar B_j\), then \(B_j = B_i^-\).
\end{proof}

Corollary~\ref{cor:augmenting-sequence-repeating-element} implies that any augmenting sequence of an $\ell$-tracking 
decomposition is finite. 

\subsection{Hanging edges and complete tracking decomposition}


All concepts defined in this subsection refers to a tracking
decomposition \(\B\)
of a graph $G$. We recall that any tracking in \(\B\)
has exactly two end-vertices, even if they coincide. For \(B\)
in \(\B\),
we denote by \(\tau(B)\)
the number of end-vertices of \(B\)
that have degree greater than~$1$ in \(\bar B\).
Thus, \(\tau(B) = 0\)
if and only if \(\bar B\)
is a path.  We observe that the same notation is used for trails (as
the meaning for both coincides).  Let
\(\tau(\B) = \sum_{B\in\B}\tau(B)\).


Let \(uv\) be an edge of $G$, and let \(B\) be the element of~\(\B\) that contains \(uv\).
If \(B = x_0x_1\cdots x_\ell\) 
with either \(x_0 = u\) and \(x_1=v\), or \(x_\ell = u\) and \(x_{\ell-1} = v\), 
then we say that \(uv\) is a  
\emph{pre-hanging} edge at \(v\) in 
\(\B\). If, additionally,  \(d_{\bar B}(u) = 1\), then we say that \(uv\) is a 
\emph{hanging}
edge at \(v\) in $\B$. We denote by \({\prehang}(v,\B)\)
(resp. \({\hang}(v,\B)\)) the number of pre-hanging (resp. hanging) edges 
at \(v\) in \(\B\). 
%
%
%
%
%
%
Let \(k\) be a positive integer.
We say that \(\B\) is \emph{\(k\)-pre-complete} 
if \({\prehang}(v,\B) > k\)
for every \(v\) in \(V(G)\). If \({\hang}(v,\B) > k\)
for every \(v\) in \(V(G)\), then we say that \(\B\) is \emph{\(k\)-complete.}


For $v$ in $V(G)$, let \(\B_{odd}(v)\) be the number of elements \(B\) of 
\(\B\) such that \(d_{\bar B}(v)\) is odd, and let \(\B_{even}(v)\) be the number of 
elements \(B=x_0\cdots x_{\ell}\) in~\(\mathcal{B}\) 
such that $x_0=x_\ell=v$. Furthermore, define \(\B(v) = \B_{odd}(v) + 2\B_{even}(v)\). 
One can see $\B(v)$ as the number of edges of \(G\) incident to~\(v\) that are 
starting edges of trackings in \(\mathcal{B}\) 
that start at \(v\), or ending edges of trackings in~$\mathcal{B}$ that end at $v$.
We note that if \(\B\) is an \(\ell\)-tracking decomposition of \(G\), 
then \(\sum_{v\in V(G)}\B(v) = 2|\B| = 
2|E(G)|/\ell\),
because each element of \(\B\) has exactly two end-vertices (counted with 
their multiplicities). The next lemma 
is the main tool in the proof of the Disentangling Lemma 
(Lemma~\ref{lemma:disentangling}).


\begin{lemma}\label{lemma:augmenting-sequence}
	Let \(k\) and \(\ell\) be positive integers
	and let \(\B\) be a \(k\)-complete \(\ell\)-tracking decomposition of a graph \(G\).
	If \(\B\) contains a \full-augmenting sequence, 
	then there is an \(\ell\)-tracking decomposition \(\B'\) of \(G\) 
	such that the following holds.
	\begin{itemize}
	 \item \(\tau(\B') < \tau(\B)\);
	 \item \(\B'(v) = \B(v)\) for every $v\in V(G)$;
	 \item $\B'$ is \(k\)-complete.
	\end{itemize}
\end{lemma}

\begin{proof}
Let \(k\), \(\ell\) and \(\B\) be as in the hypothesis of the lemma.
	 Suppose that \(S=B_1\cdots B_r\) is a \full-augmenting sequence  of
\(\B\), where \(S=B_1\cdots B_r\), and 
	\(B_i=b_0^ib_1^i\cdots b_\ell^i\) for $i=1,\ldots,r$.

	The proof is by induction on the number of elements of \(S\), denoted by $|S|$. 
	Note that by the definition of full-augmenting sequence, we have \(b^1_0\notin V(B_r)\).
	Therefore, \(|S| = r>1\).
	
	Suppose \(|S|= 2\).
        Since $S$ is a \full-augmenting sequence,
        \(b_0^1 \notin V(B_2)\)
        and, by item (ii) of Definition~\ref{def:augmenting},
        $b^2_1 = b^1_{*}$.  Let \(\bar B_1' = \bar B_1 - e_1 + f_2\)
        and \(\bar B_2' = \bar B_2 - f_2 + e_1\).
        That is, $\bar B_1'$ and $\bar B_2'$ are obtained from $\bar B_1$ and
        $\bar B_2$ by interchanging the edges $e_1$ and $f_2$. 
        Then we consider the following trackings corresponding to
        these trails: $B_1'= b_0^2X^-b_{+}^1 \cdots b_{\ell}^1$, where $X=b_0^1b_1^1 \cdots
        b_*^1$;  and $B_2'= b_0^1b_1^2b_2^2\cdots b_{\ell}^2$.
        It is easy to see that  $B_1'$ and $B_2'$ are $\ell$-trackings of $G$,  and
        furthermore, $\bar B_1\cup \bar B_2 = \bar B_1' \cup \bar
        B_2'$. 

        Let  \(\B' = \B - B_1 - B_2 + B_1' + B_2'\). Clearly, \(\B'\) is
        an $\ell$-tracking decomposition of $G$. 
	By items (i1) and (iii) of Definition~\ref{def:augmenting},
        \(d_{\bar B_1}(b_0^1) > 1\)
        and \(b_0^2 \notin V(B_1)\),
        from where we conclude that \(d_{\bar B_1'}(b_0^2) = 1\)
        and \(\tau(B_1') \leq \tau(B_1)-1\).
        Since \(b_0^1 \notin V(B_2)\),
        we have \(d_{\bar B_2'}(b_0^1) = 1\). Thus, 
        \(\tau(B_2') \leq \tau(B_2)\), and therefore 
        the following inequality holds. 
	\begin{align*}
		\tau(\B') = \tau(\B) - \tau(B_1) - \tau(B_2) + \tau(B_1') + \tau(B_2')
			  < \tau(\B).
	\end{align*}
	
	It remains to prove (for $|S|=2$) that \(\B'(v) = \B(v)\)
        for every  \(v\in V(G)\), and that $\B'$ is $k$-complete.

\begin{claim}
	\(\B'(v) = \B(v)\) for every \(v\in V(G)\).
\end{claim}

\begin{proof}
Given $v\in V(G)$ and a set 
$\mathcal{T}\subset \B$, define $\B_{odd}|_\mathcal{T}(v)$ as the number of 
elements \(B\in\mathcal{T}\) such that 
\(d_{\bar B}(v)\) is odd, and define $\B_{even}|_\mathcal{T}$ as the number of 
elements \(B=x_0\cdots x_{\ell}\) of 
\(\mathcal{T}\) such that $x_0=x_\ell=v$. Furthermore, let
$\B|_\mathcal{T}(v)=\B_{odd}|_\mathcal{T}(v) + 
2\B_{even}|_\mathcal{T}(v)$.


Let $B_{\vertices}=\{b_0^1,b_0^2,b_1^2\}$. Clearly, 
$\B(v)=\B'(v)$ for every $v\notin 
B_{\vertices}$. Let $\mathcal{T}=\{B_1,B_2\}$ and
$\mathcal{T'}=\{B_1',B_2'\}$ .
To prove that  $\B(v)=\B'(v)$ 
holds also for vertices $v\in B_{\vertices}$, it is enough to show that 
$\B|_{\mathcal{T}}(v)=\B'|_{\mathcal{T'}}(v)$, because we already know that 
$\B|_{\B\setminus\mathcal{T}}(v)=\B'|_{\B'\setminus\mathcal{T'}}(v)$. 
Recall that $\B|_{\mathcal{T}}(v)$
$\big($resp. $\B'|_{\mathcal{T'}}(v) \big)$ is the number 
edges of \(G\) that are starting edges of elements in \(\mathcal{T}\) $\big($resp. \(\mathcal{T'}\)$\big)$ that 
start at \(v\), or ending edges of elements in \(\mathcal{T}\) $\big($
resp. \(\mathcal{T'}\)$\big)$ that end at 
\(v\).  First, note that
$\B|_{\mathcal{T}}(b_0^1)=\B'|_{\mathcal{T'}}(b_0^1)$.  Indeed, the
edge $f_1=b_0^1b_1^1$ is the starting edge of $B_1$, but it is not a
starting edge of neither $B_1'$ or $B_2'$; but, on the other hand, the
starting edge of $B_2'$ starts at $b_0^1$. Thus, the number of
starting edges that starts at $b_0^1$ is the same in $\B$ and in
$\B'$. In terms of ending edges that end at $b_0^1$, the same happens:
if $e_1=b_*^1b_0^1$ is and ending edge of $B_1$, then the
ending edge of $B_1'$ (which is the reverse of $f_1$) also ends at
$b_0^1$, and if  $e_1$ is not an ending edge of $B_1$, then neither $B_1'$ or $B_2'$ has an ending edge incident to $b_0^1$. 
It is easy to see that
$\B|_{\mathcal{T}}(b_1^2)=\B'|_{\mathcal{T'}}(b_1^2)$, as $b_1^2$ is
an internal vertex of all trackings under analysis. Also,
$\B|_{\mathcal{T}}(b_0^2)=\B'|_{\mathcal{T'}}(b_0^2)$, as  the
starting edge $f_2= b_0^2b_1^2$ of $B_2$ becomes the starting edge of
$B_1'$, and no other change occurs in terms of ending edges at
$b_0^2$.
\end{proof}

	
\begin{claim}
	\(\B'\) is \(k\)-complete.
\end{claim}
 
\begin{proof}
Let us prove that \(\hang(v,\B')>k\) for every $v\in V(G)$.
Note that if \(v\neq b_1^2\), then the hanging edges at \(v\)
in \(\B'\) are the same hanging edges at \(v\) in \(\B\).
Let \(E_{\hang}\) be the set of hanging edges at \(b_1^2\) in \(\B\).
 Since $d_{\bar B_1}(b_0^1) > 1$ (by Definition~\ref{def:augmenting} (i1)), we know that 
\(e_1=b_0^1b_{*}^1=b_0^1b_1^2 \notin E_{\hang}\).
The set of hanging edges at \(b_1^2\) in \(\B'\) is \(E_{\hang}\cup 
\{b_0^1b_1^2, b_0^2b_1^2\}\) because 
\(d_{\bar B_1'}(b_0^2) = 1\) and \(d_{\bar B_2'}(b_0^1) = 1\).
Then, \(\hang(b_1^2,\B') >  \hang(b_1^2,\B)  > k \).
Therefore, \(\B'\) is \(k\)-complete.
\end{proof}

In the rest of the proof we assume that \(|S| = r> 2\). 
Suppose that the lemma holds when \(\B\) contains a \full-augmenting sequence \(S'\)
with length \(r-1\).

Since \(|S| > 2\),
by item (ii) of Definition~\ref{def:augmenting} we have
\(b_1^2 = b_{*}^1\)
and \(b_{s(2)}^2=b_0^1\)
where \(s(2) \geq 3\).
Now, consider the trackings $B_1'$ and $B_2'$ that we have
  defined in the proof for the case $|S|=2$. Let $\B''$ be the
  $\ell$-tracking decomposition as we have defined in that case (which we
  called $\B'$), that is, \(\B'' = \B - B_1 - B_2 + B_1' + B_2'\).
  In the case $|S|=2$, we had $b_0^1\notin V(B_2)$, but now we have
  that $b_0^1\in V(B_2)$, thus, in this case we can only conclude that
  \(\tau(\B'') \leq \tau(\B)\).  See Figure~\ref{fig:augmenting-a}.


The next step is to prove that \(\B''(v) = \B(v)\) for all $v\in
V(G)$. The proof follows analogously to the proof we have presented 
for the case $|S|=2$. 


Now we will prove that \(\B''\) is 
\(k\)-complete. Note that if \(v\neq b_1^2\), then the hanging edges at \(v\)
in \(\B\) are the same hanging edges at \(v\) in \(\B''\).
Now, let \(E_{\hang}(b_1^2)\) be the set of hanging edges at \(b_1^2\) in 
\(\B\).
Then the set of hanging edges at \(b_1^2\) in \(\B''\) is 
\(E_{\hang}(b_1^2)\cup \{b_1^2b_0^2\}\) because 
\(d_{\bar B_1'}(b_0^2) = 1\).
Then, \(\hang(b_1^2,\B'') \geq \hang(b_1^2,\B) > k\).
Therefore, \(\B''\) is \(k\)-complete.




Since $S$ is an augmenting sequence of $\B$, by 
Corollary~\ref{cor:augmenting-sequence-repeating-element}, every $\bar B_i$ 
appears at most twice in $S$ and if 
$\bar B_i=\bar B_j$ for $1\leq i<j\leq r$, 
then, \(B_j = B_i^-\). 
Let \(S' = C_2C_3\cdots C_r\), where, $C_2= B_2'$ and for $3\leq i\leq r$, we have
\[C_i = \begin{cases} 
B_1'^- &\mbox{if } B_i= B_1^-; \\ 
B_2'^- &\mbox{if } B_i= B_2^-; \\ 
B_i & \mbox{otherwise}.\end{cases}\]



We shall prove that $S'$ is a \full-augmenting sequence. For that, we
shall check each of the items of
Definition~\ref{def:augmenting}. Before, we make some observations: we
also denote by $s(i)$ the smallest index such that
$c_0^2 = c_{s(i)}^i=b_0^1$, for $2 \leq i \leq r$.  The vertex $c_j^i$
is the same as $b_j^i$ for $i=3,\ldots,r$ and
$j=0,\ldots,s(i)$.  We denote by $e_i^*$ and $f_i^*$ the edges
  of $C_i$ that correspond to $e_i$ and $f_i$ defined for $B_i$, that
  is, $e_i^*= c_*^ib_0^1$ and $f_i^*=c_0^ic_1^i$.\\


\noindent{\textbf{Item (i)}}: \(\bar C_2\) is not a path, \(d_{\bar
  C_2}(c^2_0) > 1\) and $c^3_0 \notin V(C_2)$. \\

Since $C_2=B_2'$, we have $c^2_0=b^1_0$. Moreover, since 
$\bar B_2'= \bar B_2 - f_2 +e_1$, the edges $e_1$ and $e_2$ are in
$\bar B_2'=C_2$ and are incident to $c_0^2$. Thus, \(d_{\bar C_2}(c^2_0) > 1\).
Now, let us prove
that $c^3_0\notin V(C_2)=V(B_2')$.  Since \(B_2 \neq B_1\),
by item (iii) of Definition~\ref{def:augmenting} (applied to \(S\)
with \(i=2\)),
we have \(b_0^3\notin V(B_2)\).
Since \(b_0^1\in V(B_2)\),
we know that \(V(B'_2)\subset V(B_2)\).
Therefore, \(b_0^3\notin V(B'_2)\).
By the construction of the elements $C_i$, we have \(c_0^3 = b_0^3\),
which implies that \(c_0^3\notin V(B'_2)=V(C_2)\).\\ 


\noindent{\textbf{Item (ii)}}: For $i=3,\ldots, r-1$, 
the element \(C_i\) contains \(c_0^2\), and \(c_1^{i} =c^{i-1}_{*}\). \\

Fix \(i\in \{3, \ldots, r\}\). Since $b^1_0\in B_i$, by the definition of $C_i$ 
we have that $c_0^2 = b^1_0\in C_i$.

We shall prove that \(c_1^{i} =c^{i-1}_{*}\). 
(a) If $C_i = B_i$ and $C_{i-1} = B_{i-1}$, then the result follows
by the definition of $C$ and the fact that item (ii) of
Definition~\ref{def:augmenting} holds for the sequence~ $S$.
(b) Suppose $C_i=B_2'^-$.  In this case, $B_i = B_2^-$, and thus
$b_1^i = b_{\ell -1}^2$.  Since $C_i=B_2'^-$, we have that
$c_1^i = b_{\ell-1}^2$.  Combining the equalities
above, we conclude that $c_1^i = b_1^i$.  (b1) If $C_{i-1} = B_{i-1}$
then $c_*^{i-1} = b_*^{i-1}$.  Thus,
$c_1^i =b_1^i = b_*^{i-1} = c_*^{i-1}$ (as the middle equality holds
because item (ii) of Definition~\ref{def:augmenting} holds for $S$).
(b2) If $C_{i-1} \neq B_{i-1}$, then $C_{i-1} = B_1'^-$. The last
equality implies that $c_*^{i-1} = b_{+}^1$ and $B_{i-1} = B_1^-$.
From the last equality, we obtain that $b_*^{i-1} = b_{+}^1$.
Combining the equalities, we get
$c_*^{i-1}= b_{+}^1 =b_*^{i-1} = b_1^i = c_1 ^i$.
 
(c) Suppose $C_i=B_1'^-$. The proof for this case is analogous to the
proof of case (b), interchanging the occurrences of index $2$ and index
$1$. We write the proof for completeness. In this case, $B_i = B_1^-$, and thus
$b_1^i = b_{\ell -1}^1$. Since $C_i=B_1'^-$, we have that
$c_1^i = b_{\ell-1}^1$.  Combining the equalities
above, we conclude that $c_1^i = b_1^i$. (c1) If $C_{i-1} = B_{i-1}$
then $c_*^{i-1} = b_*^{i-1}$.  Thus,
$c_1^i =b_1^i = b_*^{i-1} = c_*^{i-1}$. (c2) If $C_{i-1} \neq B_{i-1}$, then $C_{i-1} = B_2'^-$. The last
equality implies that $c_*^{i-1} = b_{+}^2$ and $B_{i-1} = B_2^-$.
From the last equality, we obtain that $b_*^{i-1} = b_{+}^2$.
Combining the equalities, we get
$c_*^{i-1}= b_{+}^2 =b_*^{i-1} = b_1^i = c_1 ^i$.

(d) Suppose $C_i=B_i$ and $C_{i-1}\neq B_{i-1}$. If $i=3$ then
$C_2=B_2'$, and in this case, $c_1^3=b_1^3 = b_*^2=c_*^2$. If $i>3$,
then $C_{i-1} = B_2'^-$ or $C_{i-1}=B_1'^-$. In both cases, it follows
that $c_*^{i-1} = b_* ^{i-1}$. Then using the fact that $b_1^i=
b_*^{i-1}$ (definition of $S$), it follows that $c_1^i=b_1^i=
b_*^{i-1}=c_*^{i-1}$. \\

\noindent{\textbf{Item (iii)}}: For $i=3,\ldots, r-1$,  if
\(\bar C_i \neq \bar C_h\) for every $h<i$,  then $c_0^{i+1} \notin V(C_i)$. \\

Fix $i\in\{3, \ldots, r-1\}$ and suppose \(\bar C_i \neq \bar C_{h}\) 
for every \(2\leq h < i\).
Note that $\bar C_i\neq \bar B_2'$.

First we consider the case where $\bar C_i=\bar B_1'$
or, equivalently, \(C_i = B_1'^-\).
Thus, \(B_i = B_1^-\),
and by item (iv) of Definition~\ref{def:augmenting} applied to \(S\), we have 
\(b_0^{i+1}\notin (V(\bar B_i)\cup\{b_0^2\}) =V(B_1')=V(C_i)\).
Since \(c_0^i = b_0^i\) for every \(i\geq 3\), we have \(c_0^{i+1}\notin 
V(C_i)\).

Now suppose that \(\bar C_i\neq \bar B_1'\).
Then, \(\bar B_i \neq \bar B_1\).
But we know that \(\bar C_i \neq \bar C_{h}\)
for every \( h < i\),
which implies that \(B_i \neq B_{h}\)
for every \(h <i\).
From item (iii) of Definition~\ref{def:augmenting} (applied to \(S\)),
we have that \(b_0^{i+1}\notin V(B_i)\).
Since \(c_0^{i+1}=b_0^{i+1}\),
we have \(c_0^{i+1}\notin V(B_i)\).
Since \(\bar C_i \neq \bar B_1', \bar B_2'\),
we conclude that \(C_i = B_i\), and therefore, \(c_0^{i+1}\notin V(C_i)\). \\


\noindent{\textbf{Item (iv)}}: for \(i=3,\ldots,r-1\) , if 
 \(\bar C_i = \bar C_2\) then \(c_0^{i+1}\notin V(C_i)\cup\{c_0^3\}\).\\

Fix \(i\in\{3,\ldots,r-1\}\) and suppose $\bar C_i= \bar C_2$. In this case, $C_i =B_2'^-$. 
We shall prove that 
$c_0^{i+1}\notin V(B_2')\cup\{c_0^3\}$. Note that, by the definition of 
\(C_i\), we have \(B_i = B_2^-\). Thus,
by item~(v) of Definition~\ref{def:augmenting} applied to \(S\)
with parameters \(i\) and \(h = 2\),
we have 
\begin{equation}\label{itemiv-a}
b_0^{i+1} \notin V(\bar B_2- f_2 + e_1 - e_2 + f_3).
\end{equation}
Note that 
\(\bar C_2 = \bar B_2' = \bar B_2- f_2 + e_1\).
Therefore, 
\begin{align}\label{itemiv-b}
V(\bar B_2- f_2 + e_1 - e_2 + f_3)
	= V(\bar C_2- e_2 + f_3)
	= V(\bar C_2)\cup \{b_0^3\}.
\end{align}
Recall that \(c_0^{i+1} = b_0^{i+1}\) for every \(i\geq 2\). Then, 
by~\eqref{itemiv-a}~and~\eqref{itemiv-b}, 
we have \(c_0^{i+1} \notin V(C_2)\cup\{c_0^3\}\).
\\


\noindent{\textbf{Item (v)}}:  
for \(i=3,\ldots,r-1\) , 
if \(\bar C_i = \bar C_{h}\) for some \(2<h <i\), 
then 
\(c_0^{i+1}\notin V(\bar C_i - f_h^*  + e_{h-1}^* - e_h^*+ f_{h+1}^*)\).\\

Fix $i\in\{3,\ldots,r-1\}$  and suppose that
\(\bar C_i = \bar C_{h}\) for some \(2<h<i\). 
Note that we have \(C_i = C_{h}^-\), and thus \(\bar B_i = \bar B_{h}\) and, 
by Corollary~\ref{cor:augmenting-sequence-repeating-element},
$\bar C_i\neq \bar B_1',\bar B_2'$.
By item~(v) of Definition~\ref{def:augmenting} applied to \(S\), 
we have
\begin{equation}\label{itemv-a}
b_0^{i+1} \notin  V(\bar B_h - f_h + e_{h-1}  - e_h + f_{h+1}). 
\end{equation}

Recall that, since \(i\geq 3\), we have $c_0^{i+1} = b_0^{i+1}$, $f_i^* = f_i$ and $e_i^* = e_i$. 
Therefore,  from~\eqref{itemv-a}, we have
\begin{equation*}
c_0^{i+1} \notin  V(\bar C_h - f_h^* + e_{h-1}^*   - e_h^* + f_{h+1}^*). 
\end{equation*}

We concluded the proof that $S'$ is an augmenting sequence of $\B''$. But, since $S$ is a 
\full-augmenting sequence of $\B$, we know 
that 
$b^1_0\notin V(B_r)$ (then, clearly $\bar B_r\neq \bar B_1,\bar B_2$). But since 
$c^2_0=b^1_0$, we conclude that $S'$ is a 
\full-augmenting sequence of $\B''$.

Since \(|S'| = r-1\), by the induction hypothesis, 
\(G\) admits a $k$-complete \(\ell\)-tracking decomposition \(\B'\)
such that \(\tau(\B') < \tau(\B'') 
\leq \tau(\B)\) and $\B'(v) = \B''(v) = \B(v)$
for every vertex $v$ of $G$.
\end{proof}

The following concept and lemma are important in the construction of \full-augmenting sequences.

\begin{definition}
Let \(\ell\) be a positive integer.
Let \(G\) be a graph and  \(\B\) be an \(\ell\)-tracking decomposition of \(G\).
We say that \(\B\) is \emph{feasible} if for every \(v\in V(G)\) the
following holds: if \(T\) is a 
vanilla $\ell$-trail of $G$ (not necessarily in \(\B\)) that contains \(v\) as an internal vertex, then 
there exists a hanging edge \(vw\) at \(v\) in $\B$ 
such that \(w\notin V(T)\).
\end{definition}


\begin{lemma}\label{lemma:choose-element-all}
Let \(\ell\) and \(k\) be a positive integers and \(G\) be a bipartite graph. 
If \(k\geq \lceil (\ell +1)/2\rceil\) and 
\(\B\) is an \(k\)-complete \(\ell\)-tracking decomposition of \(G\),
then \(\B\) is feasible.
\end{lemma}

\begin{proof}
  Let $\ell$, \(k\), $G$ and $\B$ be as in the hypothesis
  of the lemma. 
   Fix $v\in V(G)$ and suppose \(T\)
  is a vanilla $\ell$-trail of $G$ that contains \(v\).
  Since \(\B\)
  is \(k\)-complete,
  $\hang(v,\B)> k$.  Let
  \(vw_1,\ldots,vw_{k+1}\)
  be hanging edges at $v$ in $\B$.

	We claim that there exists an index \(1\leq i\leq k+1\) 
	such that \(w_i \notin V(T)\). 	Let \(W = \{w_1,\ldots,w_{k+1}\}\).
	Let $G=(A,B;E)$ and suppose, without loss of generality, that \(v\in A\).
	Since \(G\) is bipartite, \(W\subset B\). Furthermore, since \(T\) 
contains at most \(\ell + 1\) 
	vertices, \(|V(T)\cap B| \leq  \lceil (\ell+1)/2\rceil \leq k\). But since 
$|W|=k+1$, we 
	conclude that there exists a vertex $w\in W$ such that \(w\notin V(T)\).
\end{proof}

%
Recall that, for a tracking \(B\), 
we denote by \(\tau(B)\) the number of end-vertices 
of \(B\) that have degree greater than $1$,
and for a tracking decomposition \(\B\), 
we denote by \(\tau(\B)\) the sum \(\sum_{B\in \mathcal{B}}\tau(B)\).

\begin{lemma}\label{lemma:existence-of-sequence}
	Let \(\ell\) be a positive integer, \(G\) be a graph and 
	 \(\B\) be an \(\ell\)-tracking decomposition of $G$. If $\B$ is
         feasible and \(\tau(\B) > 0\),
	then \(\B\) contains a \full-augmenting sequence.
\end{lemma}

\begin{proof}
Let \(\ell\), \(G\) and $\B$ be as in the hypothesis of the lemma. 	
	First, let us show that \(\B\) contains an augmenting sequence.
	Since \(\tau(\B) > 0\), 
	the tracking decomposition \(\B\) contains a tracking \(B_1\) 
	that does not induce a path.
Let $B_1 = b_0^1b_1^1\cdots b_\ell^1$, where $d_{\bar B_1}(b^1_0)>1$. 
	
	Since \(\B\) is feasible, there is a hanging edge \(b_{*}^1w\) 
	at (the internal vertex) \(b_{*}^1\) in \(\B\) such that \(w\notin V(B_1)\).
	Let \(B_2\) be the element of \(\B\) that contains the edge \(b_{*}^1w\).
	Then, it is easy to verify that \(B_1B_2\) is an augmenting sequence of $\B$.
	
	Let \(S=B_1B_2\cdots B_r\) be a maximal augmenting sequence of $\B$.
	Suppose by contradiction that \(S\) is not a \full-augmenting sequence, i.e, \(b_0^1 \in V(B_r)\).

	Now we show how to obtain an element \(B_{r+1}\) of $B$ such that 
$S'=B_1\cdots B_rB_{r+1}$ is an augmenting sequence, contradicting the maximality 
of $S$. 
Since \(S\) is an augmenting sequence, item (i1) of Definition~\ref{def:augmenting} holds,
and items (ii)--(v) of Definition~\ref{def:augmenting} hold for \(i=1,\ldots,r-1\).
Since \(S\) is not a \full-augmenting sequence, \(B_r\) contains \(b_0^1\).
Our aim is to find an element \(B_{r+1}\) for which
items (iii)--(v) of Definition~\ref{def:augmenting} hold for \(i=r\).
Before continuing, note that \(b_*^i\) is a vertex of \(B_i\)
in the tracking \(P_i = b_1^ib_2^i\cdots b_{*}^i\), and
therefore \(b_*^i\) is always an internal vertex of \(B_i\) 
(because \(B_i\) contains the tracking \(P_i b_0^1\)).
Now, note that exactly one of the following holds: 
(a)~\(\bar B_r \neq \bar B_h\), for every \(h<r\);
(b)~\(\bar B_r = \bar B_1\); or
(c)~\(\bar B_r = \bar B_h\) for some \(1 < h < r\).

\begin{itemize}
\item[(a)] 
	In this case, by the feasibility of \(\B\), considering \(T=\bar B_r\)
	and \(v=b_{*}^r\), 
	there exists a hanging edge \(b_{*}^rz\) at \(b_{*}^r\) such that \(z\notin V(B_r)\).
	Let \(B_{r+1}\) be the element of \(\B\) containing \(b_{*}^rz\).
	We can suppose without loss of generality that \(z=b_0^{r+1}\)
	(otherwise, \(z = b_\ell^{r+1}\) and we choose \(B_{r+1}^-\) instead of \(B_{r+1}\)).
	Then, \(b_1^{r+1} = b_*^{r}\), \(b_0^{r+1} = z \notin V(B_r)\), and 
	item (iii) of Definition~\ref{def:augmenting} holds for \(i=r\).

\item[(b)]  
	In this case, we have \(B_r = B_1^-=b^1_\ell P_r b^1_0 P_1^-b_0^1\).
	By the feasibility of \(\B\), considering
	\(v=b_{*}^r\) and \(T=\bar B_1 - e_1 + f_2\)
	(note that \(T\) is induced by the tracking \(b_\ell^1P_rb_0^1P_1b_0^2\)), 
	there exists a hanging edge \(b_{*}^rz\) at \(b_{*}^r\) such 
that \(z\notin V(T)\).
	Note that \(V(T) = V(B_1)\cup\{b_0^2\}\).
	Let \(B_{r+1}\) be the element of \(\B\) containing 
\(b_{*}^rz\).
	As in the previous case, we may assume that \(z = b_0^{r+1} \notin V(B_r)\cup\{b_0^2\}\).
	Then, \(b_1^{r+1} = b_*^{r}\) and 
	item (iv) of Definition~\ref{def:augmenting} holds for \(i=r\).
	
\item[(c)]
	In this case, we have \(B_r = b_\ell^h P_r b_0^1 P_h^-b_0^h\).	
	Since $S$ is an augmenting sequence, by item (ii) of 
Definition~\ref{def:augmenting}, $b_1^{h}=b^{h-1}_{*}$ 
	and $b_1^{h+1}=b^{h}_{*}$. Then, since 
$b^{h-1}_{s(h-1)}=b^{h}_{s(h)}=b^1_0$, we conclude that 
	$e_{h-1}=b_1^{h}b_0^1$ and $f_{h+1}=b_{*}^{h}b_0^{h+1}$ are edges of 
\(B_{h-1}\) and \(B_{h+1}\), respectively.
	Put \(T = \bar B_r - b_1^{h}b_0^{h} + b_1^{h}b_0^1 
	- b_{*}^{h}b_0^1 + b_{*}^{h}b_0^{h+1} 
	= \bar B_r - f_h + e_{h-1} 
	- e_h + f_{h+1}\)
	(note that \(T\) is induced by the tracking \(b_\ell^hP_rb_0^1P_hb_0^{h+1}\)).
%
%
	Since $T$ is a vanilla $\ell$-trail, by the feasibility of \(\B\), considering
	\(v=b_{*}^r\) and \(T\), 
	there is a hanging edge \(b_{*}^rz\) at \(b_{*}^r\) such 
that \(z\notin V(T)\).
	Let \(B_{r+1}\) be the element of \(\B\) containing 
\(b_{*}^rz\).
	As in the previous case, we may assume that \(z = b_0^{r+1} \notin V(T)\).
	Then, \(b_1^{r+1} = b_*^{r}\) and 
	item (v) of Definition~\ref{def:augmenting} holds for \(i=r\).
\end{itemize}

	

	We just proved that there exists an element $B_{r+1}$ of $\B$ 
such that $S'=B_1B_2\cdots B_r B_{r+1}$ is 
	an augmenting sequence of \(\B\),
	a contradiction to the maximality of \(S\). 
	Therefore, \(S\) is a \full-augmenting sequence.
\end{proof}

The next result follows directly from Lemmas~\ref{lemma:augmenting-sequence}~and~\ref{lemma:existence-of-sequence}.

\begin{corollary}\label{cor:main}
	Let \(\ell\) be a positive integer and let \(G\) be a graph.
	If $\B$ is a feasible \(k\)-complete \(\ell\)-tracking decomposition of $G$
	and \(\tau(\B) > 0\),
	then there is an \(\ell\)-tracking decomposition \(\B'\) of \(G\) 
	such that the following holds.
	\begin{itemize}
	 \item \(\tau(\B') < \tau(\B)\);
	 \item \(\B'(v) = \B(v)\) for every $v\in V(G)$;
	 \item $\B'$ is \(k\)-complete.
	\end{itemize}
\end{corollary}

The next lemma, the main result of this section, combines
Lemma~\ref{lemma:choose-element-all} and Corollary~\ref{cor:main} to 
obtain \(\ell\)-path tracking decompositions
from  $\lceil(\ell +1)/2\rceil$-complete \(\ell\)-tracking decompositions.

\begin{lemma}[The Disentangling Lemma]\label{lemma:disentangling}
	Let \(\ell\) and $k$ be positive integers and let \(G\) be a bipartite graph.
	If $k\geq \lceil (\ell +1)/2\rceil$ and $\B$ is a $k$-complete 
	\(\ell\)-tracking decomposition of $G$,
	then \(G\) admits a $k$-complete
	\(\ell\)-path tracking decomposition \(\B'\)
	such that \(\B'(v) = \B(v)\) for every vertex \(v\) of \(G\).
\end{lemma}

\begin{proof}
Let \(\ell\), \(k\), \(G\) and  $\B$ be as in the hypothesis of the lemma. 
Let \(\mathbb{B}\) be the set of all \(k\)-complete 
\(\ell\)-tracking decompositions \(\B'\) of \(G\)
such that \(\B'(v) = \B(v)\) for every vertex \(v\) of \(G\).
By the hypothesis, \(\mathbb{B} \neq \emptyset\).
Let \(\tau^* = \min\{\tau(\B') \colon \B' \in \mathbb{B}\}\) 
and let \(\B_{\min}\) be an element of \(\mathbb{B}\) such that \(\tau(\B_{\min}) = \tau^*\).
If \(\tau^* = 0\), then \(\B_{\min}\) is an \(\ell\)-path tracking decomposition 
and the proof is complete. 
Then, assume \(\tau^*>0\).
By Lemma~\ref{lemma:choose-element-all},
\(\B_{\min}\) is a feasible \(\ell\)-tracking decomposition.
Since \(\tau(\B_{\min})>0\), by Corollary~\ref{cor:main}
(applied with \(k\), $\ell$, $G$ and $\B_{\min}$), 
there exists an $k$-complete \(\ell\)-tracking decomposition $\B'$ 
of $G$ such that \(\tau(\B') < \tau(\B_{\min}) = \tau^*\) and \(\B'(v) = \B(v)\) for every 
vertex \(v\) of \(G\). 
Therefore, $\B'$ is an element of $\mathbb{B}$ with $\tau(\B') < \tau(\B_{\min})$, 
a contradiction to the minimality of \(\tau^*\).
\end{proof}

\section{Factorizations}\label{section:factorizations}

The goal of this section is to show that some bipartite highly edge-connected 
graphs
admit ``well structured'' decompositions, called \emph{bifactorizations}, which 
are important structures in 
the proof of the main theorems of this paper (shown in Section~\ref{section:high}). 
The diagram of Figure~\ref{fig:outline} shows how the results of this section are related.

\subsection{Fractional factorizations}\label{section:fractional}

We extend ideas developed in~\cite{BoMoOsWa14+} in 
order to prove that some highly edge-connected bipartite 
graphs admit structured factorizations. Let us start with some definitions.

\begin{definition}[Factor]
	Let $r$ and $k$ be positive integers and \(G=(V,E)\) be a graph. Let 
\(X\subset V\) and \(F \subset E\).
	We say that \(F\) is an \emph{\((X,r,k)\)-factor} of 
\(G\) if, for every \(v\in X\), we have \(d_F(v) = (r/k)d_G(v)\).
\end{definition}

\begin{definition}[Fractional factorization]
Let \(k\) and \(\ell\) be positive integers such that $k-\ell$ is a 
positive even number. Let $G=(V,E)$ be a graph and let \(X\subset V\).
We say that a partition \(\mathcal{F}=\{M_1,\ldots,M_\ell,F_1,\ldots,F_{(k-\ell)/2}\}\) of $E$ is an 
\emph{\((X,\ell,k)\)-fractional factorization} of $G$ if the following holds.
	\begin{itemize}
		\item \(M_i\) is an \((X,1,k)\)-factor of $G$, for $1\leq i\leq \ell$;
		\item \(F_j\) is an Eulerian \((X,2,k)\)-factor of $G$, for $1\leq j\leq (k-\ell)/2$.	
	\end{itemize}
\end{definition}

Note that, if \(G\) contains an \((X,1,k)\)-factor, then \(d_G(v)\) is 
divisible by $k$ for every \(v\in X\). Therefore, 
this fact implies that, if \(G\) admits an \((X,\ell,k)\)-fractional 
factorization, 
then \(d(v)\) is divisible by \(k\) for every \(v\in X\). The next lemma is the 
core of this section.

\begin{lemma}\label{lemma:1-special-decomposition}
Let \(k\) be a positive integer.
If $G=(A,B;E)$ is a $2k$-edge-connected bipartite graph such that $d_G(v)$ is 
divisible by $2k+1$ for every 
$v\in A$, then $G$ admits an \((A,1,2k+1)\)-fractional factorization.
\end{lemma}

\begin{proof}
Let \(G=(A,B;E)\) be as in the hypothesis.
First, we want to apply Lemma~\ref{lemma:splitting-of-vertices} to obtain a 
\(2k\)-edge-connected graph $G'$ 
with maximum degree \(4k-1\).
To do this, for every vertex $v \in B$, we take integers $s_v \geq 1$ and $0 \leq r_v < 2k$ 
such that $d_G(v) = 2ks_v + r_v$.
We put $d^v_1 = 2k+r_v$ and \(d^v_2 = \cdots = d^v_{s_v} = 2k\).
Furthermore, for every 
vertex $v\in A$, we put $s_v=d_G(v)/(2k+1)$ and $d^v_i=2k+1$ for $1\leq i\leq s_v$.
By Lemma~\ref{lemma:splitting-of-vertices} 
(applied with parameters $2k$ and the integers $s_v$, $d^v_i$ ($1\leq i\leq 
s_v$) for every $v\in V(G)$), there exists a 
$2k$-edge-connected bipartite graph $G'$ obtained from \(G\) by splitting each 
vertex \(v\) of \(A\) into \(s_v\) 
vertices of degree \(2k+1\), and each vertex \(v\) of \(B\) into a 
vertex of
degree $2k+r_v < 4k$ and $s_v-1$ vertices of degree $2k$.
Let \(A'\) and \(B'\) be the set of vertices of \(G'\) obtained from 
the vertices of \(A\) and \(B\),
respectively. 
For ease of notation, if \(v \in (A' \cup B') \setminus (A \cup B)\) we also 
denote by \(v\) the vertex in \(A \cup B\) that gave rise to $v$.

The next step is to obtain a $(2k+1)$-regular multigraph $G^*$ from $G'$ by 
using lifting operations. For this, we 
will add some edges to $A'$ and remove the even-degree vertices of $B'$ by successive
applications of Mader's Lifting Theorem as 
follows.
Let \(G'_0,G'_1,\ldots,G'_\lambda\) be a maximal sequence of graphs such that \(G'_0 = G'\)
and (for \(i\geq 0\)) \(G'_{i+1}\) is the graph obtained from \(G'_i\) 
by the application of an admissible lifting 
at an arbitrary vertex \(v\) with $d_{G'}(v)\notin \{1,2,2k+1\}$.

Recall that, given any two distinct vertices of \(G'\), say \(x\) and \(y\), we denote by \(p_{G'}(x,y)\)
the maximum number of pairwise edge-disjoint paths joining \(x\) and \(y\) in \({G'}\).
We claim that \(p_{G'_i}(x,y) \geq 2k\) for any \(x,y\) in \(A'\) and every \(i\geq 0\).
Clearly, \(p_{G'_0}(x,y) \geq 2k\) holds for any \(x,y\) in \(A'\), since \(G'\) is \(2k\)-edge-connected.
Fix \(i\geq 0\) and suppose \(p_{G'_i}(x,y) \geq 2k\) holds for any \(x,y\) in \(A'\).
Let \(x,y\) be two vertices in \(A'\).
Since \(G'_{i+1}\) is a graph obtained from \(G'_i\) by the application 
of an admissible lifting at a vertex \(v\) in \(B'\), 
we have \(p_{G'_{i+1}}(x,y) \geq p_{G'_i}(x,y) \geq 2k\).

We claim that, if \(v\in B'\)
then \(d_{G'_\lambda}(v) \in \{2,2k+1\}\).
Suppose, for a contradiction, that there is a vertex \(v\) in \(B'\) 
such that \(d_{G'_\lambda}(v) \notin \{2,2k+1\}\).
Note that \(d_{G'_i}(u) \geq d_{G'_{i+1}}(u)\geq 2\)
for every \(u\in V(G')\) and every \(0\leq i\leq \lambda\).
Since \(d_{G'}(u) \leq 4k-1\) for every \(u\in V(G')\),
we have \(2\leq d_{G'_i}(u)\leq 4k-1\) for every \(0\leq i\leq \lambda\).
Therefore, \(2\leq d_{G'_\lambda}(v)\leq 4k-1\).
Since \(d_{G'_\lambda}(v) \leq 4k-1\),
and for any two neighbors $x$ and $y$ of \(v\) we have \(p_{G'_\lambda}(x,y)\geq 2k\),
Lemma~\ref{lemma:non-cut-vertex} implies that \(v\) is not a cut-vertex of \(G'_\lambda\).
Then, by Mader's Lifting Theorem (Theorem~\ref{thm:Mader}) applied to \(G'_\lambda\),
there is an admissible lifting at \(v\).
Therefore, \(G'_0,G'_1,\ldots,G'_\lambda\) is not maximal, a contradiction.

In \(G'_\lambda\) the set $B'$ may have some vertices of degree $2$. 
For every such vertex $v$, if $u$ and $w$ are the neighbours of $v$, we apply a \(uw\)-lifting at $v$, 
and remove the vertex $v$, i.e, we perform a supression of \(v\).
Let \(G^*\) be the graph obtained by applying this process to all vertices of degree $2$ in $B'$.
Note that the number of pairwise edge-disjoint paths joining two distinct vertices of \(A'\)
remains the same, i.e, \(p_{G^*}(x,y) \geq p_{G_\lambda}(x,y)\geq 2k\) for every \(x,y\) in \(A'\).
Clearly, the 
set of vertices of $G^*$ that belong to $B'$ is an independent set;  
we denote it by $B^*$  (eventually, \(B^* = \emptyset\)). 
Furthermore, every vertex in \(B^*\) has degree \(2k+1\).

\begin{claim}
\(G^*\) is \(2k\)-edge connected.
\end{claim}
\begin{proof}
Let \(Y \subset V(G^*)\).
Suppose there is at least one vertex \(x\) of \(A'\) in \(Y\)
and at least one vertex \(y\) of \(A'\) in \(V(G^*)-Y\).
Since there are at least \(2k\) edge-disjoint paths joining \(x\) to \(y\),
there are at least \(2k\) edges with vertices in both \(Y\) and 
\(V(G^*)-Y\).
Now, suppose that \(A'\subset Y\) (otherwise \(A'\subset V(G^*)-Y\) and
we take \(V(G^*)-Y\) instead of \(Y\)), and then \(V(G^*)-Y \subset B^*\).
Since \(B^*\) is an independent set, all edges with a vertex in \(V(G^*)-Y\) 
must have the other vertex in \(A'\).
Since every vertex in \(B^*\) has degree \(2k+1\), 
there are at least \(2k+1\) edges with vertices in both \(Y\) and \(V(G^*)-Y\). 
\end{proof}

We conclude that $G^*$ is a 
$2k$-edge-connected $(2k+1)$-regular multigraph with vertex-set $A'\cup B^*$, 
where $B^*$ is an independent set.

Since every vertex of \(G^*\) has odd degree, $|V(G^*)|$ is even.
Thus, \(G^*\) is a \(2k\)-edge-connected 
\((2k+1)\)-regular multigraph of even order. 
By Theorem~\ref{thm:Kano}, \(G^*\) contains a perfect matching \(M^*\).
Since the multigraph \(J^* = G^* - M^*\)  
is \(2k\)-regular, Theorem~\ref{thm:Petersen} implies that \(J^*\) admits a decomposition 
into \(2\)-factors, say \(F^*_1,\ldots,F^*_k\).
Therefore, \(M^*,F^*_1,\ldots,F^*_k\) is a partition of \(E(G^*)\).

Now, let us get back to the bipartite graph $G$. 
Let \(xy\) be an edge of \(G^*\).
If $x\in A'$ and $y\in B^*$, 
then \(xy\) corresponds to an edge of \(G\). On the other hand, if \(x,y\in A'\),
then there is a vertex \(v_{xy}\) of \(B'\) and two edges 
\(xv_{xy}\) and \(v_{xy}y\) in $E(G')$. Furthermore,
\(xy\) was obtained by an \(xy\)-lifting 
at \(v_{xy}\) (either by an application of Mader's Lifting 
Theorem or by the supression of vertices of degree $2$).
Then, each edge of \(G^*\) represents an edge of \(G\) or a \(2\)-path
in \(G\) such that the internal vertices of these \(2\)-paths are 
always 
in \(B\).
For every edge $xy\in E(G^*)$, define \(f(xy) = \{xy\}\) if $x\in A'$ and $y\in B^*$,
and \(f(xy) = \{xv_{xy},v_{xy}y\}\) if \(x,y\in A'\). Note that $f(xy)\subset E(G)$ for every edge $xy\in E(G^*)$.
For a set \(S\subset E(G^*)\), put \(f(S) = \cup_{e\in S} f(e)\).
The partition of \(E(G^*)\) into \(M^*,F^*_1,\ldots,F^*_k\) induces 
a partition of \(E(G)\) into \(M = f(M^*)\) and \(F_i = f(F^*_i)\) for $1\leq i\leq k$.

We will prove that \(\{M,F_1,\ldots,F_k\}\) is an \((A,1,2k+1)\)-fractional 
factorization. Fix $i\in\{1,\ldots,k\}$. We will show that \(M\) is an \((A,1,2k+1)\)-factor of $G$ 
and \(F_i\) is an Eulerian \((A,2,2k+1)\)-factor of $G$.
Let \(v\) be a vertex of \(A\) in \(G\) and put \(d'(v) = d(v)/(2k+1)\). Then, we know that 
\(v\) is represented by \(d'(v)\) vertices in \(G^*\).
Since \(M^*\) is a perfect matching in \(G^*\), there are \(d'(v)\) edges
of \(M\) entering \(v\) and, since \(F^*_i\) is a \(2\)-factor in \(G^*\),
there are \(2d'(v)\) edges of \(F_i\) incident to \(v\).
Finally, since \(F^*_i\) is Eulerian, the set
\(F_i\) is Eulerian, concluding the proof. 
\end{proof}

\begin{corollary}\label{corollary:even-special-decomposition}
Let \(k\) be a positive integer.
If $G=(A,B;E)$ is a $32k$-edge-connected bipartite 
graph such 
that $d_G(v)$ is divisible by $2k+2$ for every 
$v\in A$, then $G$ admits an \((A,2,2k+2)\)-fractional factorization.	
\end{corollary}

\begin{proof}
Let \(k\) and \(G=(A,B;E)\) be as in the hypothesis.
We claim that \(G\) contains an \((A,1,2k+2)\)-factor \(F\) such that
\(G-F\) is \(2k\)-edge-connected (note that this implies that $d_{G-F}(v)$ is 
divisible by $2k+1$ for every $v\in A$).

Since $G$ is $32k=16k\lceil(2k+2)/(2k+1)\rceil$-edge-connected, 
by Lemma~\ref{lem:barat-part-3} (applied with parameters $2k+1$, $m=2k$ and $r=1$), 
the graph $G$ admits a 
decomposition into graphs \(G_k\) and \(G_r\) such that \(G_k\) is 
\(2k\)-edge-connected and \(d_{G_k}(v) = 
\big((2k+1)/(2k+2)\big)d_G(v)\), and \(d_{G_r}(v) = \big(1/(2k+2)\big)d_G(v)\) 
for every \(v\in A\).
Therefore, \(E(G_r)\) is an \((A,1,2k+2)\)-factor.

By Lemma~\ref{lemma:1-special-decomposition}, \(G_k\) admits an 
\((A,1,2k+1)\)-fractional factorization 
\(\mathcal{F}\). Therefore, since \(d_{G_k}(v) = 
\big((2k+1)/(2k+2)\big)d_G(v)\) for every $v\in A$, we conclude that 
\(\mathcal{F} + E(G_r)\) is an 
\((A,2,2k+2)\)-fractional factorization of \(G\).

\end{proof}

\subsection{Bifactorizations}\label{sec:bip-factors}


To obtain a decomposition of highly edge-connected bipartite graphs
$G$ into paths of a fixed length $\ell$, we will combine fractional
factorizations to obtain first an \(\ell\)-tracking decomposition.
More specifically, we decompose $G$ into 
graphs $G_1$ and $G_2$ and then we combine 
a fractional factorization of \(G_1\) with a fractional factorization of \(G_2\). 
This process, called bifactorizations, is defined as follows.

\begin{definition}[Bifactorization]\label{def:bip-factor}
Let \(k\) and \(\ell\) be positive integers such that $k-\ell$ is a positive 
even number, and let \(G=(A_1,A_2;E)\) be a bipartite graph. 

Let \(\mathcal{F}_1,\mathcal{F}_2\) be families of subsets of \(E\)
and put \(G_i = G[\cup_{F\in\mathcal{F}_i}F]\), for \(i=1,2\).
We say that \(\mathbb{F}=(\mathcal{F}_1,\mathcal{F}_2)\) is an 
\emph{\((\ell,k)\)-bifactorization of \(G\)}
if the following holds.
\begin{itemize}
 \item [(i)] \(\{G_1,G_2\}\) is a decomposition of \(G\) ;
 \item [(ii)] \(\mathcal{F}_i\) is an \((A_i,\ell,k)\)-fractional 
factorization of $G_i$, for $1\leq i\leq 2$.
\end{itemize}
If \(G\) admits an \((\ell,k)\)-bifactorization,
we say that \(G\) is \((\ell,k)\)-bifactorable.
\end{definition}

The next concept  
will be used to guarantee 
that \(G_1\) and \(G_2\) have a sufficiently 
large minimum degree.

\begin{definition}[Strong bifactorization]
Let \(k\) and \(\ell\) be positive integers.
Let \(G = (A_1,A_2;E)\) be a bipartite graph that admits an
\((\ell,k)\)-bifactorization \(\mathbb{F} = (\mathcal{F}_1,\mathcal{F}_2)\).
Let \(E_i = \bigcup_{F\in\mathcal{F}_i} F\) for $1\leq i\leq 2$. 
Let $p=1$ if $\ell$ is odd, and $p=2$ if $\ell$ is even. 
We say that \(\mathbb{F}\) is \emph{strong}
if \(d_{E_i}(v)\geq (k/p)((k/p)+p)\) for every 
\(v\) in \(A_i\)
for $1\leq i\leq 2$.
If \(G\) admits a strong \((\ell,k)\)-bifactorization, we say that \(G\) is 
\emph{strongly \((\ell,k)\)-bifactorable}.
\end{definition}

For ease of notation, if \(F\) belongs to either \(\mathcal{F}_1\) or 
\(\mathcal{F}_2\), then
we say that \(F\) is an element of \(\mathbb{F}\).
In what follows, we give sufficient conditions for a bipartite graph to be 
strongly
bifactorable.

\begin{lemma}\label{lemma:odd-strong-factorization}
	Let \(k\) be a positive integer. Let \(r = (2k+1)(2k+2)\).
	If \(G\) is a  \(2(6k + 2r+1)\)-edge-connected 
	bipartite graph such that \(|E(G)|\) is divisible by \(2k+1\),
	then \(G\) is strongly \((1,2k+1)\)-bifactorable.
\end{lemma}

\begin{proof}
	Let $k$, $r$ and $G=(A,B;E)$ be as in the hypothesis.
	By Lemma~\ref{lem:tight-good-initial-decomposition} (applied with 
\(2k+1\) and $r$), the graph \(G\) can be decomposed into two spanning edge-disjoint 
\(r\)-edge-connected 
	graphs \(G_1\) and \(G_2\) such that all vertices of \(A\) have degree 
divisible by \(2k+1\)
	in \(G_1\), and all vertices of \(B\) have degree divisible by 
\(2k+1\) 
in \(G_2\).  But since $r\geq 2k$, by Lemma~\ref{lemma:1-special-decomposition} (applied with \(k\)), 
	we conclude that \(G_1\) admits 
	an \((A,1,2k+1)\)-fractional factorization
	and \(G_2\) admits a \((B,1,2k+1)\)-fractional factorization.
	Therefore, \(G\) is \((1,2k+1)\)-bifactorable.
	Since \(G_1\) and \(G_2\) are \(r\)-edge-connected,
	we have \(d_{G_1}(v) \geq (2k+1)(2k+2)\) for every \(v\in A\),
	and \(d_{G_2}(v) \geq (2k+1)(2k+2)\) for every \(v\in B\), from 
where we conclude that \(G\) is strongly 
\((1,2k+1)\)-bifactorable.
\end{proof}

The proof of the next lemma can be easily obtained by replacing Lemma~\ref{lemma:1-special-decomposition} with
Corollary~\ref{corollary:even-special-decomposition} in the proof of Lemma~\ref{lemma:odd-strong-factorization}.

\begin{lemma}\label{lemma:even-strong-factorization}
	Let \(k\) be a positive integer.
	Let \(r = \max\{32k,(k+1)(k+3)\}\).
	If \(G\) is a  \(2(6k + 2r + 4)\)-edge-connected bipartite graph
	 such that \(|E(G)|\) is divisible by \(2k+2\),
	then \(G\) is strongly \((2,2k+2)\)-bifactorable.
\end{lemma}

	

	
\section{Decomposition into paths of odd length}~\label{section:odd-paths}

We present now a definition which is central to what follows.
Before that, we recall that given 
an \(\ell\)-tracking decomposition \(\B\) of a graph $G$, 
$\B(v)$ denotes the number of edges of \(G\) incident to \(v\) that are 
starting edges of trackings in \(\B\) 
that start at \(v\), or ending edges of trackings in $\B$ that end at $v$.

\begin{definition}[Balanced tracking decomposition -- odd case]
Let \(k\) be a positive integer.
Let \(G = (A,B;E)\) be a bipartite graph 
that admits a \((1,2k+1)\)-bifactorization 
\(\mathbb{F}=(\mathcal{F}_1,\mathcal{F}_2)\),
and let \(G_i = G\big[\bigcup_{F\in\mathcal{F}_i}F\big]\) for \(i=1,2\).
Let \(M_1\) be the \((A,1,2k+1)\)-factor of \(\mathbb{F}\) and \(M_2\) be the
\((B,1,2k+1)\)-factor of \(\mathbb{F}\).
We say that 
a \((2k+1)\)-tracking decomposition \(\B\) of \(G\) is \emph{\(\mathbb{F}\)-balanced}
if the following holds.
	\begin{itemize}
		\item \(\B(v) = d_{G_1}(v)/(2k+1) + d_{M_2}(v)\) for every \(v\in A\);
		\item \(\B(v) = d_{G_2}(v)/(2k+1) + d_{M_1}(v)\) for every \(v\in B\).
	\end{itemize}
\end{definition}

Our aim in this section is to prove
Theorem~\ref{thm:odd-path-decomposition}, which states that one may obtain an
$\mathbb{F}$-balanced \((k+1)\)-complete
path tracking decomposition from a strong
$(1,2k+1)$-bifactorization~$\mathbb{F}$. The proof of
Theorem~\ref{thm:odd-path-decomposition} is by induction on $k$. The
base of the induction is precisely the statement of the next lemma, whose proof can be seen
in~Thomassen~\cite{Th08b} (we present it for completeness). 



\begin{lemma}\label{lemma:3-PD}
	Let \(G\) be a bipartite graph that admits a 
	\((1,3)\)-bifactorization \(\mathbb{F}\).
	Then \(G\) admits an \(\mathbb{F}\)-balanced $3$-path tracking decomposition.
\end{lemma}

\begin{proof}
	Let \(G = (A,B;E)\) be a bipartite graph that admits a 
	\((1,3)\)-bifactorization \(\mathbb{F}= (\mathcal{F}_1,\mathcal{F}_2)\) 
and put \(\mathcal{F}_i = \{M_i,F_i\}\) for \(i = 1,2\).
	Let \(\mathcal{C}_1\) be the set of components of \(G[F_1]\).
	Let \(T\) be an element of \(\mathcal{C}_1\)
	and \(B_T = a_0b_0a_1b_1\cdots a_sb_sa_0\) be a tracking of \(T\),
	where \(a_i \in A\) and \(b_i \in B\), for \(1\leq i\leq s\).
	We have that \(\B'_T = \{a_ib_ia_{i+1}\colon 0\leq i\leq s\}\), taking \(a_{s+1} = a_0\),
	is a \(2\)-tracking decomposition of \(T\)
	in which every tracking has its end-vertices in \(A\).
	Therefore, \(\B'_1 = \cup_{T \in \mathcal{C}_1}  \B'_T\)
	is a \(2\)-tracking decomposition of \(G[F_1]\)
	in which every tracking has its end-vertices in \(A\).
	Analogously, \(G[F_2]\) admits a \(2\)-tracking decomposition \(\B'_2\)
	in which every tracking has its end-vertices in \(B\).
	
	Let \(G_i = G[M_i\cup F_i]\) for \(i = 1,2\).
	Note that, since \(M_1\) is an \((A,1,3)\)-factor and \(F_1\) is an 
\((A,2,3)\)-factor
	of \(G_1\), we have \(d_{F_1}(v) = (2/3)d_{G_1}(v) = 2d_{M_1}(v)\),
	for every vertex \(v\) in \(A\).
	Note that the number of trackings in \(\B'_1\) that end at a vertex \(v\) 
	equals \(\frac{1}{2}d_{F_1}(v) = d_{M_1}(v)\).
%
	Thus, we can extend each tracking \(B\) of \(\B'_1\) by adding 
	an edge of \(M_1\) to the start vertex of \(B\),
	obtaining a \(3\)-path tracking decomposition \(\B_1\) of \(G_1\).
	Analogously, we can extend each tracking \(T\) of \(\B'_2\) by adding
	an edge of \(M_2\) to the starting vertex of \(T\),
	obtaining a \(3\)-path tracking decomposition \(\B_2\) of \(G_2\).

	Put \(\B = \B_1\cup\B_2\). 
	If \(v\) is a vertex of \(A\), then the number of paths having \(v\) as end-vertex
	is exactly \(d_{F_1}(v)/2 + d_{M_2}(v)\).
	Therefore,
	we have \(\B(v) =  d_{F_1}(v)/2 +  d_{M_2}(v)  = d_{G_1}(v)/3 + d_{M_2}(v)\).
	Analogously, we have \(\B(v) = d_{G_2}(v)/3 + d_{M_1}(v)\)
	for every vertex \(v\) in \(B\).
	Thus, \(\B\) is an \(\mathbb{F}\)-balanced $3$-path tracking decomposition of \(G\).
	
\end{proof}


In the next lemma, we show how pre-completeness is related to completeness
of odd tracking decompositions.

\begin{lemma}\label{lemma:odd-complete}
	Let \(k\) be a positive integer and let 
	\(G\) be a bipartite graph that admits a  
	\((1,2k+1)\)-bifactorization \(\mathbb{F}\).
	If \(G\) admits an \(\mathbb{F}\)-balanced 
\((2k+1)\)-pre-complete
	$(2k+1)$-tracking decomposition, 
	then \(G\) admits an \(\mathbb{F}\)-balanced \((k+1)\)-complete 
	$(2k+1)$-tracking decomposition.
\end{lemma}

\begin{proof}
	Let \(k\) be a positive integer and let \(G = (A,B;E)\) be a bipartite 
graph that admits a 
	\((1,2k+1)\)-bifactorization \(\mathbb{F}\). Let
$\mathcal{F}$ be the set of all $\mathbb{F}$-balanced 
	\((2k+1)\)-pre-complete $(2k+1)$-tracking decompositions of $G$. 
	Now let $\B$ be a tracking decomposition in $\mathcal{F}$ such that \(\sum_{v\in V(G)} 
\hang(v,\B)\) is maximum over all tracking decompositions in $\mathcal{F}$.
	We claim that \(\B\) is a \((k+1)\)-complete tracking decomposition,
	i.e,  \(\hang(v,\B) > k+1\) for each vertex \(v\) of~\(G\).

	Suppose, for a contradiction, that \(\B\) is not \((k+1)\)-complete.
	Then there is a vertex \(v\) of \(G\) such that \(\hang(v,\B) \leq k+1\).
	Suppose, without loss of generality, that \(v\) is a vertex of \(A\).
	Since \(\B\) is \((2k+1)\)-pre-complete, \(\prehang(v,\B) \geq 2k+2\).
	Thus, there are at least \(k+1\) pre-hanging edges at \(v\) 
	that are not hanging edges at \(v\), say \(x_1v,\ldots,x_{k+1}v\).
	Let \(T_1 = y_0y_1\cdots y_{2k+1}\) be the element of \(\B\) that contains \(x_1v\),
	where, without loss of generality, \(y_0 = x_1\) and \(y_1 = v\).
	The vertices of \(T_1\) in \(B\) are \(y_0,y_2,\ldots, y_{2k}\). 
	First, we will show that $x_i\notin V(T_1)$ for some \(2\leq i \leq k+1\). 
	Since \(y_2\) is in \(B\) and \(y_0\) is the end-vertex of \(T_1\) in 
\(B\),
	the vertex \(y_2\) is not an end-vertex of \(T_1\). 
	Since \(y_1y_2 \in E(T_1)\), the edge \(y_1y_2\) is not a pre-hanging 
edge at \(v\).
	Then, \(y_2 \neq x_i\) for every \(2\leq i\leq k+1\).
	Therefore, \(|\{y_4, y_6,\ldots,y_{2k}\}| = k-1\) and 
\(|\{x_2,\ldots,x_{k+1}\}| = k\), from where we conclude that
	for at least one \(i\), we have \(x_i \notin V(T_1)\).
	
	Now let \(T_i = z_0z_1\cdots z_{2k+1}\) be the element of \(\B\) that contains \(x_iv\).
	Suppose, without loss of generality, that \(z_1 = v\) and \(z_0 = x_i\).
	Let \(T'_1 = vy_1\cdots y_{2k+1}\) and \(T'_i = y_0z_1\cdots z_{2k+1}\),
	and let \(\B' = \B - T_1 - T_i + T'_1 + T'_i\).
	We note that \(\bar T'_1 = \bar T_1 - x_1v + x_iv\) 
	and \(\bar T'_i = \bar T_i -x_iv + x_1v\).
	Since \(x_i \notin V(T_1)\), we have \(d_{T'_1}(x_i) = 1\), which 
	implies that \(x_iv\) is a hanging edge of \(\B'\) at~$v$. Therefore, 
	\(\sum_{v\in V(G)} \hang(v,\B') > \sum_{v\in V(G)} \hang(v,\B)\).

	We claim that \(\B'\) is \(\mathbb{F}\)-balanced.
	Indeed, since the set of pre-hanging edges at \(v\) is the same in \(\B\) 
	and \(\B'\), we have \(\B(v) = \B'(v)\) for every \(v\) in \(V(G)\).
	Therefore, \(\B'\) is an \hbox{\(\mathbb{F}\)-balanced} \((2k+1)\)-pre-complete
	tracking decomposition such that 
	\(\sum_{v\in V(G)} \hang(v,\B') > \sum_{v\in V(G)} \hang(v,\B)\),	
	a contradiction to the choice of \(\B\).
	\end{proof}

Now we are ready to use the Disentangling Lemma 
(Lemma~\ref{lemma:disentangling})
to obtain a path decomposition from the previous tracking decomposition.

\begin{lemma}\label{lemma:odd-path-dec}
	Let \(k\) be a positive integer and let \(G\) be a bipartite graph that 
admits a 
	\((1,2k+1)\)-bifactorization \(\mathbb{F}\).
	If \(G\) admits an \(\mathbb{F}\)-balanced \((k+1)\)-complete 
	$(2k+1)$-tracking decomposition, 
	then \(G\) admits an \(\mathbb{F}\)-balanced \((k+1)\)-complete 
	$(2k+1)$-path tracking decomposition.
\end{lemma}

\begin{proof}
	Let \(k\) be a positive integer and let \(G = (A,B;E)\) be a bipartite 
graph that admits a \((1,2k+1)\)-bifactorization \(\mathbb{F}\).
	Suppose \(G\) admits an \(\mathbb{F}\)-balanced 
\((k+1)\)-complete
	$(2k+1)$-tracking decomposition \(\B\).
	By Lemma~\ref{lemma:disentangling} (applied with \(\ell = 2k+1\)),
	the graph \(G\) admits a \((k+1)\)-complete \((2k+1)\)-path tracking decomposition \(\B'\)
	such that \(\B'(v) = \B(v)\) for every vertex \(v\) of \(G\). 
	Therefore, \(\B'\) is an \(\mathbb{F}\)-balanced 
	\((k+1)\)-complete $(2k+1)$-path tracking decomposition.
\end{proof}

Now we have the tools needed for the proof of the next result, which is the 
main result of this section.

\begin{theorem}\label{thm:odd-path-decomposition}
	Let \(k\) be a positive integer. If \(G\) is a bipartite graph that 
admits a strong  \((1,2k+1)\)-bifactorization \(\mathbb{F}\),
	then \(G\) admits an \(\mathbb{F}\)-balanced $(2k+1)$-path tracking decomposition.
\end{theorem}

\begin{proof}
The proof is by induction on \(k\).
By Lemma~\ref{lemma:3-PD}, the statement is true for \(k=1\).
Thus, suppose \(k>1\),
and let \(G = (A_1,A_2;E)\) be a bipartite graph that admits a strong  
\((1,2k+1)\)-bifactorization \(\mathbb{F} = (\mathcal{F}_1,\mathcal{F}_2)\).
We claim that \(G\) admits an \(\mathbb{F}\)-balanced 
\((2k+1)\)-pre-complete 
$(2k+1)$-tracking decomposition.
Let \(\mathcal{F}_1 = \{M_1,F_{1,1},\ldots,F_{1,k}\}\) 
and \(\mathcal{F}_2 = \{M_2,F_{2,2},\ldots,F_{2,k}\}\),
and let \(G_i = G\big[\bigcup_{F\in\mathcal{F}_i} F\big]\) for \(i = 1,2\).
Hereafter, fix $i \in \{1,2\}$.

Define \(d^*(v) = 
d_{G_i}(v)/(2k+1)\) for every vertex \(v\) in \(A_i\).
Note that \(d_{F_{i,j}}(v) = 2d^*(v) = 2d_{M_i}(v)\) for every vertex \(v\) in 
\(A_i\) and \(1\leq j\leq k\). 
Let \(O_{F_{i,k}}\) be a Eulerian 
orientation of \(F_{i,k}\).
Let \(F_{i,k} = F_{i,k}^{\forw} \cup F_{i,k}^{\back}\), where 
\(F_{i,k}^{\forw}\) is the set of 
edges
of \(F_{i,k}\) leaving \(A_i\) in \(O_{F_{i,k}}\),
and \(F_{i,k}^{\back}\) is the set of edges of \(F_{i,k}\) entering
\(A_i\)
in \(O_{F_{i,k}}\).
Note that, since \(O_{F_{i,k}}\) is an Eulerian orientation,
\(d_{F_{i,k}^{\forw}}(v) = d_{F_{i,k}^{\back}}(v) = d_{F_{i,k}}(v)/2\) for 
every 
vertex \(v\in A_i\).

Let \(G' = G - M_1 - M_2 - F_{1,k}^{\forw} - F_{2,k}^{\forw}\),
and define \(\mathcal{F}'_i = \{F_{i,k}^{\back},F_{i,1},\ldots, F_{i,k-1}\}\).
Let \(G'_i = G\big[\bigcup_{F\in\mathcal{F}'_i} F\big]\).
Note that \(G'_i = G_i - M_i - F_{i,k}^{\forw}\).
Then, for every vertex \(v\in A_i\), we have
\begin{equation}\label{odd-d}
d_{G'_i}(v) = d_{G_i}(v) - 2d^*(v) 
	    = (2k+1)d^*(v) - 2d^*(v) 
	    = (2k-1)d^*(v).
\end{equation}

\begin{claim}
\(\mathbb{F}' = (\mathcal{F}'_1,\mathcal{F}'_2)\) is a strong 
\((1,2k-1)\)-bifactorization of \(G'\).
\end{claim}

\begin{proof}

To prove this claim, we shall prove the following.
\begin{itemize}
	\item[(i)] \(F_{i,k}^{\back}\) is an \((A_i,1,2k-1)\)-factor of 
\(G'_i\);
	\item[(ii)] \(F_{i,j}\) is an Eulerian \((A_i,2,2k-1)\)-factor 
			of \(G'_i\) for \(1\leq j\leq k-1\);
	\item[(iii)] \(d_{G'_i}(v) \geq (2k-1)(2k)\) for every vertex \(v\in 
A_i\).
\end{itemize}

To prove items (i) and (ii), first 
note that, for every $v\in A_i$, we have \(d_{F_{i,k}^{\back}}(v) = 
d^*(v)\) and \(d_{F_{i,j}}(v) = 2d^*(v)\) for every \(1\leq 
j\leq k-1\). By~\eqref{odd-d}, we conclude that \(F_{i,k}^{\back}\) is a 
\((A_i,1,2k-1)\)-factor of 
\(G'_i\) and \(F_{i,j}\) is an \((A_i,2,2k-1)\)-factor of \(G'_i\). 
Since \(\mathbb{F}\) is a \((1,2k+1)\)-bifactorization,
\(F_{1,j}\) and \(F_{2,j}\) are Eulerian graphs for 
\(1\leq j\leq k-1\).

It remains to prove item (iii).
Since \(d_{G'_i}(v) = (2k-1)d^*(v)\) and \(d^*(v) = d_{G_i}(v)/(2k+1)\) for 
every vertex \(v\in A_i\),
we have \(d_{G'_i}(v) = \frac{2k-1}{2k+1}d_{G_i}(v)\) for every \(v\in A_i\).
Since \(\mathbb{F}\) is a strong \((1,2k+1)\)-bifactorization, 
\(d_{G_i}(v) \geq (2k+1)(2k+2)\) for every \(v\in A_i\).
Therefore, \(d_{G'_i}(v) \geq (2k-1)(2k+2) > (2k-1)2k\) for every \(v\in 
A_i\).
\end{proof}

Since 
\(\mathbb{F}'\) is a strong \((1,2k-1)\)-bifactorization of \(G'\),
by the induction hypothesis, 
\(G'\) admits an \(\mathbb{F}'\)-balanced $(2k-1)$-path 
tracking decomposition \(\B'\).

Since \(\B'\) is a \(\mathbb{F}'\)-balanced $(2k-1)$-path tracking decomposition, we have
	\begin{itemize}
		\item \(\B'(v) = d_{G'_1}(v)/(2k-1) + d_{F_{2,k}^{\back}}(v)\) 
for every 
\(v\in A_1\);
		\item \(\B'(v) = d_{G'_2}(v)/(2k-1) + d_{F_{1,k}^{\back}}(v)\) 
for every 
\(v\in A_2\).
	\end{itemize}
Now we want to extend each tracking in \(\B'\) to obtain 
a \((2k+1)\)-tracking decomposition of \(G\).
For that, we add edges from  
\(E(G) - E(G')\) at the end-vertices of the trackings in $\B'$.
For each \(v\in A_1\) (\(v\in A_2\)), let \(S_v\) be the set of edges of 
\(M_1\cup F_{2,k}^{\forw}\) (\(M_2\cup F_{1,k}^{\forw}\)) that are incident 
to \(v\).
Note that for each edge \(e\) in \(E(G) - E(G')\) 
there is exactly one \(v\in V(G)\) such that \(e\in S_v\).
Then, \(\bigcup_{v\in V(G)} S_v = E(G) - 
E(G')\).
Therefore, if we prove that \(\B'(v) = |S_v|\), then we can extend every 
tracking $B$ in $\B'$ by adding one edge at each end-vertex of $B$.

\begin{claim}
	\(\B'(v) = |S_v|\), for every $v\in V(G)$.
\end{claim}

\begin{proof}
First, note that, since \(F_{i,k}\) is Eulerian, we have 
\(d_{F_{2,k}^{\back}}(v) = 
d_{F_{2,k}^{\forw}}(v)\)
for every \(v\in A_1\), and \(d_{F_{1,k}^{\back}}(v) = 
d_{F_{1,k}^{\forw}}(v)\)
for every \(v\in A_2\).

For every \(v\in A_i\) and every \(1\leq j\leq k-1\), 
we have \(d_{G'_i}(v)/(2k-1) = d_{F_{i,j}}(v)/2 = d^*(v)\).
Now, recall that for each \(v\in A_i\), we have \(d_{M_i}(v) = 
d^*(v)\). Therefore, for every \(v\in A_1\), we have 
%
$$\B'(v) = d_{G'_1}(v)/(2k-1) + d_{F_{2,k}^{\back}}(v) 
			 = d^*(v) + d_{F_{2,k}^{\forw}}(v) 
			 = d_{M_1}(v)+ d_{F_{2,k}^{\forw}}(v) 
			 = |S_v|.$$

Similarly, we have $\B'(v)=|S_v|$ for every $v\in A_2$.
\end{proof}

As we have seen, we can extend every tracking \(B\) of 
\(\B'\) by adding one edge at
each end-vertex of \(B\). Let $\B$ be the tracking decomposition obtained by this 
extension. Since $\B'$ is a $(2k-1)$-path tracking decomposition and we added edges 
only at the end-vertices of these trackings, 
$\B$ is composed of trackings of vanilla \((2k+1)\)-trails. Furthermore, since we added all edges in 
\(E(G)-E(G')\), $\B$ is a decomposition of $G$.

\begin{claim}
\(\B\) is \(\mathbb{F}\)-balanced.
\end{claim}

\begin{proof}
Fix \(x_0\in A_1\). First we will prove that  \(\B(x_0) \leq 
d_{F_{1,k}^{\forw}}(x_0) + d_{M_2}(x_0)\). If there is no tracking \(T = 
x_0x_1\cdots x_{2k+1}\) in \(\B\),
where \(x_0x_1\) is an edge of \(E(G) - E(G')\), then $\B(x_0)=0$. 
For every such tracking $T$, by the construction of $\B$, we know that \(x_0x_1\) is an 
element of \(S_{x_1}\).
Since \(x_1\) is a vertex of \(A_2\), we have \(S_{x_1} \subset M_2\cup 
F_{1,k}^{\forw}\).
Therefore,
\begin{equation*}
\B(x_0) \leq d_{F_{1,k}^{\forw}}(x_0) + d_{M_2}(x_0).
\end{equation*}

Now we will prove that \(\B(x_0) \geq d_{F_{1,k}^{\forw}}(x_0) + 
d_{M_2}(x_0)\). Note that, if \(x_0x_1\) is an edge of \(M_2\cup 
F_{1,k}^{\forw}\)
that is incident to \(x_0\) in \(A_1\) \big(these are the only edges of $G$ that 
can contribute to $\B(x_0)$\big), 
then, by the construction of $\B$, there is a tracking \(Q' = x_1\cdots x_{2k}\) in 
\(\B'\) of a path
such that the tracking \(Q = x_0x_1\cdots x_{2k+1}\) (of a vanilla trail) belongs to 
\(\B\).
Therefore \(\B(x_0) \geq d_{F_{1,k}^{\forw}}(x_0) + d_{M_2}(x_0)\), from where 
we conclude that
\(\B(v) = d_{F_{1,k}^{\forw}}(v) + d_{M_2}(v)\) for every \(v\) in \(A_1\).
Thus, for every vertex \(v\) in \(A_1\) we have 
\[
	\B(v)	= d_{F_{1,k}^{\forw}}(v) + d_{M_2}(v)
			= d^*(v) + d_{M_2}(v)
			= d_{G_1}(v)/(2k+1) + d_{M_2}(v).
\]
Analogously, we have \(\B(v) = d_{G_2}(v)/(2k+1) + d_{M_1}(v)\)
for each vertex \(v\) in \(A_2\).
Thus, \(\B\) is an \(\mathbb{F}\)-balanced $(2k+1)$-tracking decomposition.
\end{proof}

\begin{claim}
\(\B\) is \((2k+1)\)-pre-complete.
\end{claim}

\begin{proof}
Let \(v\in A_i\). We shall prove that \(\prehang(v,\B) > 2k+1\).
Note that, by the construction of $\B$, the set of pre-hanging edges at \(v\) in 
\(\B\) is exactly \(S_v\).
Then, \(\prehang(v,\B) = |S_v| = \B'(v)\).
Since \(\B'\) is balanced, \(\B'(v) \geq d_{G'_i}(v)/(2k-1)\).
Therefore,
\begin{equation*}
	\prehang(v,\B)	= \B'(v) 
					\geq d_{G'_i}(v)/(2k-1) 
					= d^*(v) 
					= d_{G_i}(v)/(2k+1).
\end{equation*} 
Since \(\mathbb{F}\) is a strong \((1,2k+1)\)-bifactorization of \(G\),
we have \(d_{G_i}(v) \geq (2k+1)(2k+2)\), from where we conclude that 
\(\prehang(v,\B) \geq 2k+2\).
Therefore, \(\B\) is a \((2k+1)\)-pre-complete tracking decomposition.
\end{proof}
Now we are able to conclude the proof.
By Lemma~\ref{lemma:odd-complete}, \(G\) admits an 
\(\mathbb{F}\)-balanced
\((k+1)\)-complete $(2k+1)$-tracking decomposition.
By Lemma~\ref{lemma:odd-path-dec}, \(G\) admits an 
\(\mathbb{F}\)-balanced
\((k+1)\)-complete $(2k+1)$-path tracking decomposition.
Therefore, \(G\) admits an \(\mathbb{F}\)-balanced $(2k+1)$-path tracking decomposition.
\end{proof}

\section{Decomposition into paths of even length}\label{section:even-paths}
The technique used in this section is analogous to the one used in 
Section~\ref{section:odd-paths}. 
The results are similar, but to deal with paths of even length some adjustments were necessary.

\begin{definition}[Balanced tracking decomposition -- even case]
Let \(k\) be a positive integer.
Let \(G = (A,B;E)\) be a bipartite graph 
that admits a \(\big(2,2(2k+2)\big)\)-bifactorization 
\(\mathbb{F}=(\mathcal{F}_1,\mathcal{F}_2)\),
and let \(G_i = G\big[\bigcup_{F\in\mathcal{F}_i}F\big]\) for \(i=1,2\).
Let \(M_1,N_1\) be the \(\big(A,1,2(2k+2)\big)\)-factors of \(\mathbb{F}\) and 
\(M_2,N_2\) be the
\(\big(B,1,2(2k+2)\big)\)-factors of \(\mathbb{F}\).
We say that a \((2k+2)\)-tracking decomposition \(\B\) of $G$ into is
\emph{\(\mathbb{F}\)-balanced} if the following holds.
	\begin{itemize}
		\item \(\B(v) = d_{G_1}(v)/(2k+2) + d_{M_2}(v)
			+ d_{N_2}(v)\) for every \(v\in A\);
		\item \(\B(v) = d_{G_2}(v)/(2k+2) + d_{M_1}(v)
			+ d_{N_1}(v)\) for every \(v\in B\).
	\end{itemize}
\end{definition}

Our aim is to prove Theorem~\ref{thm:even-path-decomposition}, which
guarantees that one may obtain an $\mathbb{F}$-balanced \((2k+2)\)-complete
$(2k+2)$-path tracking decomposition from a strong
$\big(2,2(2k+2)\big)$-bifactorization $\mathbb{F}$. First, we show
that from a \((2,4)\)-bifactorization we may obtain a balanced $2$-path tracking decomposition.

\begin{lemma}\label{lemma:2-PD}
	If \(G\) is a bipartite graph that admits a 
	\((2,4)\)-bifactorization~\(\mathbb{F}\),
	then \(G\) admits an \(\mathbb{F}\)-balanced $2$-path decomposition.
\end{lemma}

\begin{proof}
	Let \(G = (A,B;E)\) be a bipartite graph 
	that admits a \((2,4)\)-bifactorization
	\(\mathbb{F} = (\mathcal{F}_1,\mathcal{F}_2)\) and put \(\mathcal{F}_i 
= 
\{M_i,N_i,F_i\}\) for \(i = 1,2\).
	Let \(O_{F_i}\) be an Eulerian orientation of \(G[F_i]\), for \(i = 1,2\).
	Let \(\mathcal{C}_1\) be the set of components of \(G[F_1]\).
	Let \(T\) be an element of \(\mathcal{C}_1\) and \(B_T = a_0b_0a_1b_1\cdots a_sb_s\) 
	be a tracking of \(T\), where \(a_i \in A\), and \(b_i \in B\), for \(1\leq i\leq s\).
	We have that \(\B'_T = \{a_ib_ia_{i+1}\colon 0\leq i\leq s\}\),
	taking \(a_{s+1} = a_0\), is a \(2\)-tracking decomposition of \(T\)
	in which every tracking has its end-vertices in \(A\).
	Therefore, \(\B'_1 = \cup_{T\in\mathcal{C}_1}\B'_T\) is a 
	\(2\)-tracking decomposition of \(G[F_1]\) in which every tracking has its end-vertices in \(A\).
	Analogously \(G[F_2]\) admits a \(2\)-tracking decomposition \(\B'_2\) in which every tracking
	has its end-vertices in \(B\).

	Let \(v\) be a vertex in \(A\).
	Since \(M_1\) and \(N_1\) are \((1,1,4)\)-factors of \(G\) we have \(d_{M_1}(v) = d_{N_1}(v)\).
	Thus, the number of edges in \(M_1\cup N_1\) incident to \(v\) is even,
	and we can decompose the edges in \(M_1\cup N_1\) incident to \(v\) 
	into \(2\)-paths such that each path has its end-vertices in \(B\).
	Taking any tracking of each of these paths, we obtain a \(2\)-tracking decomposition \(\B''_1\) 
	of the edges in \(M_1\cup N_1\) 
	such that each path has its end-vertices in \(B\).
	Analogously, there is a \(2\)-tracking decomposition \(\B''_2\) of the edges in \(M_2\cup N_2\) 
	such that each path has its end-vertices in \(A\).
	
	Let \(\B = \B'_1\cup\B'_2\cup\B''_1\cup\B''_2\).
	Note that only the paths in \(\B'_1\) and in \(\B''_2\) have end-vertices
	in~\(A\),
	and analogously
	only the paths in \(\B'_2\) and in \(\B''_1\) have end-vertices in \(B\).
	Therefore, if \(v\) is a vertex of \(A\), 
	then \(\B(v) = \B'_1(v) + \B''_2(v) 
		= d_G(v)/2 + d_{M_2}(v) + d_{N_2}(v)\),
	and	if \(v\) is a vertex of \(B\), 
	then \(\B(v) = \B'_2(v) + \B''_2(v) 
		= d_G(v)/2 + d_{M_1}(v) + d_{N_1}(v)\).
	Thus, \(\B\) is an \(\mathbb{F}\)-balanced $2$-path tracking decomposition of \(G\).
	
\end{proof}


In the next lemma, we show how pre-completeness is related to completeness
of even tracking decompositions.  

\begin{lemma}\label{lemma:even-complete}
	Let \(k\) be a positive integer.
	Let \(G\) be a bipartite graph that admits a  
	\(\big(2,2(2k+2)\big)\)-bifactorization \(\mathbb{F}\).
	If \(G\) admits an \(\mathbb{F}\)-balanced \((2k+3)\)-pre-complete $(2k+2)$-tracking decomposition, 
	then \(G\) admits an \(\mathbb{F}\)-balanced \((k+2)\)-complete $(2k+2)$-tracking decomposition.
\end{lemma}

\begin{proof}
	Let \(k\) be a positive integer.
	Let \(G = (A,B;E)\) be a bipartite graph that admits a 
	\((2,2(2k+2))\)-bifactorization \(\mathbb{F}\).
	Let $\mathcal{F}$ be the set of all \(\mathbb{F}\)-balanced 
	\((2k+3)\)-pre-complete $(2k+2)$-tracking decompositions of $G$. 
	Now let \(\B\) be a tracking decomposition in $\mathcal{F}$ 
	such that  \(\sum_{v\in V(G)} \hang(v,\B)\) is maximum over all 
	tracking decompositions in $\mathcal{F}$.
	We claim that \(\B\) is a \((k+2)\)-complete tracking decomposition,
	i.e, \(\hang(v,\B) > k+2\) for each vertex \(v\) of \(G\).
	
	Suppose, for a contradiction, that \(\B\) is not \((k+2)\)-complete.
	Then there is a vertex \(v\) of \(G\) such that \(\hang(v,\B) \leq k+2\).
	Suppose, without loss of generality, that \(v\) is a vertex of \(A\).
	Since \(\B\) is \((2k+3)\)-pre-complete, \(\prehang(v,\B) \geq 2k+4\).
	Thus, there are at least \(k+2\) pre-hanging edges at \(v\) 
	that are not hanging edges at \(v\), say \(x_1v,\ldots,x_{k+2}v\).
	Let \(T_1 = y_0y_1\cdots y_{2k+2}\) be the element of \(\B\) that contains \(x_1v\),
	where, without loss of generality, \(y_0 = x_1\) and \(y_1 = v\).
	The vertices of \(T_1\) in \(B\) are \(y_0,y_2,\ldots, y_{2k+2}\).
	First, we will show that for some $x_i\notin V(T_1)$ for some \(2\leq i \leq k+2\). 
	Since \(y_2\) is in \(B\) and \(y_0\) is the end-vertex of \(T_1\) in \(B\), 
	the vertex \(y_2\) is not an end-vertex of \(T_1\). 
	Since \(y_1y_2 \in E(T_1)\), the edge \(y_1y_2\) is not a pre-hanging edge at \(v\).
	Then, \(y_2 \neq x_i\) for every \(2\leq i\leq k+2\). 
	Therefore, \(|\{y_4, y_6,\ldots,y_{2k+2}\}| = k\) and \(|\{2,\ldots,k+2\}| = k+1\), 
	from where we conclude that for at least one \(i\), we have \(x_i \notin V(T_1)\).
	
	Now let \(T_i = z_0z_1\cdots z_{2k+2}\) be the element of \(\B\) that contains \(x_iv\).
	Suppose, without loss of generality, that \(z_1 = v\) and \(z_0 = x_i\).
	Let \(T'_1 = vy_1\cdots y_{2k+1}\) and \(T'_i = y_0 z_1\cdots z_{2k+2}\)
	and let \(\B' = \B - T_1 - T_i + T'_1 + T'_i\).
	We note that \(\bar T'_1 = \bar T_1 - x_1v + x_iv\) and \(\bar T'_i = \bar T_i -x_iv + x_1v\).
	Since \(x_i \notin V(T_1)\), we have \(d_{T'_1}(x_i) = 1\), which 
implies that \(x_iv\) is a hanging edge of \(\B'\) at \(v\).
	Therefore, \(\sum_{v\in V(G)} \hang(v,\B') > 
			\sum_{v\in V(G)} \hang(v,\B)\).
			
	We claim that \(\B'\) is \(\mathbb{F}\)-balanced.
	Indeed, since the set of pre-hanging edges at~\(v\) is the same in 
\(\B\) and \(\B'\),
	we have \(\B(v) = \B'(v)\) for every \(v\) in \(V(G)\).
	Therefore, \(\B'\) is an \(\mathbb{F}\)-balanced 
\((2k+3)\)-pre-complete
	decomposition such that 
	\(\sum_{v\in V(G)} \hang(v,\B') > 
			\sum_{v\in V(G)} \hang(v,\B)\),
	a contradiction to the choice of \(\B\).
\end{proof}


As we did for the odd paths, we use the Disentangling Lemma to obtain a path decomposition.

\begin{lemma}\label{lemma:even-path-dec}
	Let \(k\) be a positive integer and let \(G\) be a bipartitte graph that admits a 
	\(\big(2,2(2k+2)\big)\)-bifactorization \(\mathbb{F}\).
	If \(G\) admits an \(\mathbb{F}\)-balanced \((k+2)\)-complete 
	vanilla $(2k+2)$-tracking decomposition, 
	then \(G\) admits an \(\mathbb{F}\)-balanced \((k+2)\)-complete 
	$(2k+2)$-path tracking decomposition.
\end{lemma}

\begin{proof}
	Let \(k\) be a positive integer and let \(G = (A,B;E)\) be a 
	bipartite graph that admits a 
	\(\big(2,2(2k+2)\big)\)-bifactorization \(\mathbb{F}\).
	Suppose $G$ admits an \(\mathbb{F}\)-balanced 
	\((k+2)\)-complete
	$(2k+2)$-tracking decomposition $\B$.
	By Lemma~\ref{lemma:disentangling} (applied with \(\ell = 2k+2\)),
	the graph \(G\) admits a \((k+2)\)-complete tracking decomposition \(\B'\)
	such that \(\B'(v) = \B(v)\)
	for every vertex \(v\) of \(G\).
	Therefore, \(\B'\) is an \(\mathbb{F}\)-balanced 
	$(k+2)$-complete $(2k+2)$-tracking decomposition.
\end{proof}

Now we are ready to prove the main result of this section.

\begin{theorem}\label{thm:even-path-decomposition}
	Let \(k\) be a positive integer. 
	If \(G\) is a bipartite graph that 
admits a strong 
	\(\big(2,2(2k+2)\big)\)-bifactorization \(\mathbb{F}\),
	then \(G\) admits an \(\mathbb{F}\)-balanced $(2k+2)$-path 
tracking decomposition.
\end{theorem}

\begin{proof}
The proof is by induction on \(k\).
By Lemma~\ref{lemma:2-PD}, the statement is true for \(k = 0\).
Thus, suppose \(k>0\),
and let \(G = (A_1,A_2;E)\) be a bipartite graph that admits a strong 
\(\big(2,2(2k+2)\big)\)-bifactorization 
\(\mathbb{F} = (\mathcal{F}_1,\mathcal{F}_2)\).
We claim that \(G\) admits an \(\mathbb{F}\)-balanced 
\((2k+3)\)-pre-complete
vanilla $(2k+2)$-trail decomposition. Let \(\mathcal{F}_1 = 
\{M_1,N_1,F_{1,1},\ldots,F_{1,2k+1}\}\) 
and \(\mathcal{F}_2 = \{M_2,N_2,F_{2,1},\ldots,F_{2,2k+1}\}\), and let \(G_i = 
G\big[\bigcup_{F\in\mathcal{F}_i} F\big]\) for \(i = 1,2\). 
Hereafter, fix \(i\in\{1,2\}\).

Define \(d^*(v) = d_{G_i}(v)/\big(2(2k+2)\big)\) for every vertex $v\in A_i$.
Note that \(d_{F_{i,j}}(v) = 2d^*(v) = 2d_{M_i}(v) = 2d_{N_i}(v)\) for every 
vertex \(v\) in \(A_i\)
and \(1\leq j\leq 2k+1\).
For \(j \in \{2k,2k+1\}\), let \(O_{F_{i,j}}\) be an Eulerian 
orientation of \(G[F_{i,j}]\).
Let \(F_{i,j} = F_{i,j}^{\forw} \cup F_{i,j}^{\back}\), where 
\(F_{i,j}^{\forw}\) is the set of edges
of \(F_{i,j}\) leaving $A_i$ in \(O_{F_{i,j}}\), and 
\(F_{i,j}^{\back}\) 
is the set of edges
of \(F_{i,j}\) entering $A_i$ in~\(O_{F_{i,j}}\).

Let \(G' = G - M_1 - N_1 - M_2 - N_2 - F_{1,2k}^{\forw} - F_{1,2k+1}^{\forw} 
	- F_{2,2k}^{\forw} - F_{2,2k+1}^{\forw}\),
and let \(\mathcal{F}'_i = 
\{F_{i,2k}^{\back},F_{i,2k+1}^{\back},F_{i,1},\ldots, 
F_{i,2k-1}\}\).
Let \(G'_i = G\big[\bigcup_{F\in\mathcal{F}'_i} F\big]\).
Note that \(G'_i = G_i - M_i - N_i - F_{i,2k}^{\forw} - F_{i,2k+1}^{\forw}\).
Then, for every $v\in A_i$, we have
\begin{equation}\label{even-d}
d_{G'_i}(v) = d_{G_i}(v) - 4d^*(v)
	    = 2(2k+2)d^*(v) - 4d^*(v)
	    = 2(2k)d^*(v).
\end{equation}

\begin{claim}
\(\mathbb{F}' = (\mathcal{F}'_1,\mathcal{F}'_2)\) 
is a strong \(\big(2,2(2k)\big)\)-bifactorization of \(G'\).
\end{claim}

\begin{proof}
To prove this claim, we shall prove the following.
\begin{itemize}
	\item[(i)] \(F_{i,2k}^{\back}\) and \(F_{i,2k+1}^{\back}\) 
			are \(\big(A_i,1,2(2k)\big)\)-factors of \(G'_i\);
	\item[(ii)] \(F_{i,j}\) is an Eulerian \(\big(A_i,2,2(2k)\big)\)-factor 
			of \(G'_i\) for \(j = 1,\ldots,2k-1\);
	\item[(iii)] \(d_{G'_i}(v) \geq (2k)(2k+2)\) for every vertex \(v\in 
A_i\).\\
\end{itemize}

To prove items (i) and (ii), first note that, for every $v\in A_i$, we have
 \(d_{F_{i,2k}^{\back}}(v) = d_{F_{i,2k+1}^{\back}}(v) = d^*(v)\) 
and \(d_{F_{i,j}}(v) = 2d^*(v)\) for every \(1\leq j\leq 2k-1\). 
By~\eqref{even-d}, we conclude that \(F_{i,2k}^{\back}\) and 
\(F_{i,2k+1}^{\back}\) are \(\big(A_i,1,2(2k)\big)\)-factors of \(G'_i\), and 
\(F_{i,j}\) is an \(\big(A_i,2,2(2k)\big)\)-factor of \(G'_i\).
Since \(\mathbb{F}\) is a \(\big(2,2(2k+2)\big)\)-bifactorization,
\(F_{1,j}\) and \(F_{2,j}\) are Eulerian graphs for 
\(1\leq j\leq 2k-1\).

It remains to prove item (iii). 
Since \(d_{G'_i}(v) = 2(2k)d^*(v)\) and \(d^*(v) = d_{G_i}(v)/\big(2(2k+2)\big)\) 
for every vertex \(v\in A_i\),
we have \(d_{G'_i}(v) = \frac{2(2k)}{2(2k+2)}d_{G_1}(v)\) for every \(v\in 
A_i\).
Since \(\mathbb{F}\) is a strong \(\big(2,2(2k+2)\big)\)-bifactorization, 
we have \(d_{G_i}(v) \geq (2k+2)(2k+4)\) for every \(v\in A_i\).
Thus \(d_{G'_i}(v) \geq (2k)(2k+4) > (2k)(2k+2)\) for every \(v\in A_i\).
\end{proof}

Since \(\mathbb{F}'\) is a strong \(\big(2,2(2k)\big)\)-bifactorization of \(G'\),
by the induction hypothesis, \(G'\) admits an \(\mathbb{F}'\)-balanced 
$2k$-path tracking decomposition \(\B'\).

Since \(\B'\) is an \(\mathbb{F}'\)-balanced $2k$-path tracking decomposition, we have 
	\begin{itemize}
		\item \(\B'(v) = d_{G'_1}(v)/2k
			+ d_{F_{2,2k}^{\back}}(v)
			+ d_{F_{2,2k+1}^{\back}}(v)\) for every \(v\in A_1\);
		\item \(\B'(v) = d_{G'_2}(v)/2k
			+ d_{F_{1,2k}^{\back}}(v)
			+ d_{F_{1,2k+1}^{\back}}(v)\) for every \(v\in A_2\).
	\end{itemize}
Now we want to extend each $2k$-path of \(\B'\) to obtain a \((2k+2)\)-tracking decomposition of $G$. 
For that, we add edges from 
\(E(G)-E(G')\) at the end-vertices of the trackings in $\B'$.
For each vertex \(v\in A_1\) ($v\in A_2$), let \(S_v\) be the set of edges 
of \(M_1\cup N_1 \cup F_{2,2k}^{\forw}\cup F_{2,2k+1}^{\forw}\) (\(M_2\cup N_2 
\cup F_{1,2k}^{\forw}\cup F_{1,2k+1}^{\forw}\)) incident to \(v\).
Note that for each edge \(e\) in \(E(G)-E(G')\) there is exactly 
one \(v\in V(G)\) such that \(e\in S_v\).
Then, \(\bigcup_{v\in V(G)} S_v=E(G) - E(G')\). Therefore, if we prove that 
\(\B'(v) = |S_v|\), then we can extend every tracking $B$ of $\B'$ by adding 
one edge at each end-vertex of $B$.

\begin{claim}
	 \(\B'(v) = |S_v|\), for every $v\in V(G)$.
\end{claim}

\begin{proof}
First, note that, since \(F_{2,2k}\) and \(F_{2,2k+1}\) are Eulerian, 
we have \(d_{F_{2,2k}^{\back}}(v) = d_{F_{2,2k}^{\forw}}(v)\) 
and \(d_{F_{2,2k+1}^{\back}}(v) = d_{F_{2,2k+1}^{\forw}}(v)\)
for every vertex \(v\) in \(A_1\),
and since \(F_{1,2k}\) and \(F_{1,2k+1}\) are Eulerian, we have 
\(d_{F_{1,2k}^{\back}}(v) = d_{F_{1,2k}^{\forw}}(v)\) and 
\(d_{F_{1,2k+1}^{\back}}(v) = d_{F_{1,2k+1}^{\forw}}(v)\)
for every vertex \(v\) in \(A_2\).

For every \(v\in A_i\) and every \(1\leq j\leq 2k-1\), 
we have \(d_{G'_i}(v)/(2(2k)) = d_{F_{i,j}}(v)/2 = d^*(v)\).
Now, recall that for each vertex \(v\in A_i\), we have \(d_{M_i}(v) = 
d_{N_i}(v) = d^*(v)\).
Therefore, for every $v\in A_1$, we have 
\begin{align*}
	 \B'(v) & = d_{G'_1}(v)/(2k) + d_{F_{2,2k}^{\back}}(v) + 
d_{F_{2,2k+1}^{\back}}(v) \\
			& = 2d^*(v) + d_{F_{2,2k}^{\forw}}(v) + 
d_{F_{2,2k+1}^{\forw}}(v)\\
			& = d_{M_1}(v) + d_{N_1}(v) + d_{F_{2,2k}^{\forw}}(v) + 
d_{F_{2,2k+1}^{\forw}}(v)\\
			& = |S_v|.
\end{align*}
Similarly, we have $\B'(v)=|S_v|$ for every $v\in A_2$.
\end{proof}

We have shown that every tracking \(B\) of 
\(\B'\) can be extended by adding one edge at each of its end-vertices.
Let \(\B\) be the decomposition obtained with these extensions.
Analogously to the odd case, we conclude that \(\B\) is a \((2k+2)\)-tracking decomposition of~\(G\).

\begin{claim}
\(\B\) is \(\mathbb{F}\)-balanced.
\end{claim}

\begin{proof}

Let \(x_0\) be a vertex in \(A\). First we will prove that \(\B(x_0) \leq 
d_{F_{1,2k}^{\forw}}(x_0) + d_{F_{1,2k+1}^{\forw}}(x_0) 
+ d_{M_2}(x_0) + d_{N_2}(x_0)\).
If there is no tracking \(T = x_0x_1\cdots x_{2k+2}\) in $\B$, where $x_0x_1$ is an 
edge of \(E(G)-E(G')\), then $\B(x_0)=0$. 
For every such tracking $T$, by the 
construction of $\B$, we know that $x_0x_1$ is an element of $S_{x_1}$.
Since \(x_1\) is a vertex of \(A_2\), we have 
that \(S_{x_1} \subset M_2\cup N_2 \cup F_{1,2k}^{\forw} \cup 
F_{1,2k+1}^{\forw}\).
Therefore,
\begin{equation*}
\B(x_0) \leq d_{F_{1,2k}^{\forw}}(x_0) + d_{F_{1,2k+1}^{\forw}}(x_0) 
+ d_{M_2}(x_0) + d_{N_2}(x_0).
\end{equation*}
Now we will prove that $\B(x_0)\geq d_{F_{1,2k}^{\forw}}(x_0) + 
d_{F_{1,2k+1}^{\forw}}(x_0) + d_{M_2}(x_0) + d_{N_2}(x_0)$.
Note that if \(x_0x_1\) is an edge of 
\(M_2\cup N_2 \cup F_{1,2k}^{\forw} \cup F_{1,2k+1}^{\forw}\)
that is incident to \(x_0\) in $A_1$ \big(these are the only edges of $G$ that 
can contribute to $\B(x_0)$\big), then, by the construction of $\B$, there is a 
tracking \(Q' = x_1\cdots x_{2k+2}\) of a path such that the tracking  \(Q = x_0x_1\cdots 
x_{2k+2}\)\(x_0\) (of a vanilla trail) belongs to $\B$.
Therefore, \(\B(x_0) = d_{F_{1,2k}^{\forw}}(x_0) + d_{F_{1,2k+1}^{\forw}}(x_0) 
				+ d_{M_2}(x_0) + d_{N_2}(x_0)\).
Thus, for every vertex \(v\in A_1\) we have 
\begin{align*}
	\B(v)	&= |F_{1,2k}^{\forw}(v)| + |F_{1,2k+1}^{\forw}(v)| + |M_2(v)| + 
|N_2(v)| \\
			&= 2d^*(v) + d_{M_2}(v) + d_{N_2}(v) \\
			&= d_{G_1}(v)/(2k+2) + d_{M_2}(v) + d_{N_2}(v).
\end{align*}
Analogously, we have \(\B(v) = d_{G_2}(v)/(2k+2) + d_{M_1}(v) + d_{N_1}(v)\)
for each vertex \(v\in A_2\).
Thus, \(\B\) is an \(\mathbb{F}\)-balanced vanilla $(2k+2)$-trail decomposition.
\end{proof}

\begin{claim}
\(\B\) is \((2k+3)\)-pre-complete.
\end{claim}

\begin{proof}
Let  \(v\in A_i\).
We shall prove that \(\prehang(v,\B) > 2k+3\).
Note that, by the construction of $\B$, the set of pre-hanging edges at \(v\) 
in \(\B\)
is exactly \(S_v\).
Then, \(\prehang(v,\B) = |S_v| = \B'(v)\).
Since \(\B'\) is balanced, \(\B'(v) \geq d_{G'_i}(v)/(2k)\).
Therefore,
\begin{equation*}
	\prehang(v,\B)	= \B'(v) 
					\geq d_{G'_i}(v)/(2k) 
					= 2d^*(v) 
					= d_{G_i}(v)/(2k+2).
\end{equation*} 
Since \(\mathbb{F}\) is a strong \(\big(2,2(2k+2)\big)\)-bifactorization of \(G\),
we have \(d_{G_i}(v) \geq (2k+2)(2k+4)\), from where we conclude that  
\(\prehang(v,\B) \geq 2k+4\).
Therefore, \(\B\) is a \((2k+3)\)-pre-complete vanilla trail decomposition.
\end{proof}

Now, analogously to the odd case, using Lemmas~\ref{lemma:even-complete} and~\ref{lemma:even-path-dec}, we conclude that \(G\) admits an \(\mathbb{F}\)-balanced $(2k+2)$-path tracking decomposition.
\end{proof}

\section{Decomposition of highly edge-connected graphs into 
paths}\label{section:high}

In this section we put together the results of 
Section~\ref{section:factorizations}
and Theorem~\ref{thm:odd-path-decomposition}
(resp. Theorem~\ref{thm:even-path-decomposition}) and prove 
Conjecture~\ref{conj:dec-bip} for paths of odd (resp. even) length.

\subsection{Paths of odd length}

\begin{theorem}\label{thm:high-odd}
	Let \(k\) be a positive integer and let \(r = (2k+1)(2k+2)\). If 
\(G\) is a \(2(6k + 2r+1)\)-edge-connected bipartite graph such that \(|E(G)|\) 
is divisible by \(2k+1\), 
	then \(G\) admits a decomposition into \((2k+1)\)-paths.
\end{theorem}

\begin{proof}
	Let \(k\) be a positive integer and let \(G\) be a bipartite graph such 
that \(|E(G)|\) is divisible by \(2k+1\). Let \(r = (2k+1)(2k+2)\).
	Suppose \(G\) is a \(2(6k + 2r+1)\)-edge-connected graph.
	By Lemma~\ref{lemma:odd-strong-factorization},
	\(G\) is strongly \((1,2k+1)\)-bifactorable.
	Let \(\mathbb{F}\) be a strong \((1,2k+1)\)-bifactorization of \(G\).
	By Theorem~\ref{thm:odd-path-decomposition}, 
	\(G\) admits an \(\mathbb{F}\)-balanced $(2k+1)$-path decomposition.
\end{proof}

\subsection{Paths of even length}

The proof of this case is slightly different from the odd case.
First, we prove that Conjecture~\ref{conj:dec-bip} is equivalent to the 
following conjecture.
\begin{conjecture}\label{conj:dec-bip2}
For each tree $T$, there exists a positive integer
$k''_T$ such that, if $G$ is a $k''_T$-edge-connected bipartite graph
and $2|E(T)|$ divides $|E(G)|$, then \(G\) admits a \hbox{\(T\)-decomposition.}
\end{conjecture}
We prove Conjecture~\ref{conj:dec-bip2} restricted to the case 
\(T\)  is a path of even length.
The following result states the equivalence of 
Conjecture~\ref{conj:dec-bip2} 
and Conjecture~\ref{conj:dec-bip}.

\begin{theorem}\label{theorem:conj-equivalence2}
Let $T$ be a tree with \(\ell\) edges, \(\ell \geq 3\). The following two statements are
equivalent.
\begin{itemize}
\item[(i)] There exists a positive integer $k''_T$ such that, if $G$ is a 
$k''_T$-edge-connected 
		bipartite graph and $2|E(T)|$ divides $|E(G)|$, then $G$ admits 
a $T$-decomposition.
\item[(ii)] There exists a positive integer $k'_T$ such that, if $G$ is a 
$k'_T$-edge-connected bipartite graph and $|E(T)|$ divides $|E(G)|$, then $G$ admits 
a 
$T$-decomposition.
\end{itemize}
	Furthermore, 
	\(
		k'_T\leq 2(k''_T + \ell).
	\)
\end{theorem}

\begin{proof}
	It suffices to prove that statement (i) implies statement (ii).
	Suppose statament (i) is true.
	Let \(G\) be an \(2(k''_T + \ell)\)-edge-connected graph such 
	that \(|E(G)|\) is divisible by \(\ell\).
	If \(|E(G)|\) is divisible by \(2\ell\),
	then \(G\) satisfies the conditions of statement (i),
	and therefore admits a \(T\)-decomposition.
	Thus, we may assume that \(|E(G)| = 2r\ell + \ell\) for some positive 
integer \(r\).
	Clearly, \(G\) contains a copy \(T'\) of \(T\).
	By Theorem~\ref{thm:edge-disjoint-spanning-trees}, \(G\) contains at 
least
	\(k''_T + \ell\) edge-disjoint spanning trees.
	Since \(T'\) has \(\ell\) edges, \(T'\) intercepts at most \(\ell\) of 
	these spanning trees,
	and the graph \(G' = G - T'\) contains at least \(k''_T\) of 
these spanning trees.
	Thus, \(G'\) is \(k''_T\)-edge-connected.
	Note that \(|E(G')| = |E(G)| - \ell= 2r\ell\).
	By statement (i), \(G'\) admits a decomposition \(\D'\) into copies of 
\(T\).
	Thus, \(\D = \D' + T'\) is a decomposition of \(G\) into copies of 
\(T\).
\end{proof}

\begin{theorem}\label{thm:high-even}
	Let \(k\) be a positive integer and let \(r = 
\max\{32(2k+1),(2k+2)(2k+4)\}\). If \(G\) is a bipartite 
\(2(12k + 2r + 10)\)-edge-connected graph such that \(|E(G)|\) is divisible by 
\(4k+4\),
	then \(G\) admits a decomposition into \((2k+2)\)-paths.
\end{theorem}

\begin{proof}
	Let \(k\) be a positive integer and let \(G\) be a bipartite graph 
	such that \(|E(G)|\) is divisible by \(4k+4=2((2k+1)+2)\). Let \(r = 
\max\{32(2k+1),(2k+2)(2k+4)\}\). Suppose that $G$ is a \(2(12k + 2r + 
10)\)-edge-connected 	graph, i.e, $G$ is $2(6(2k+1)+2r+4)$-edge-connected.
	
	By Lemma~\ref{lemma:even-strong-factorization} (applied with \(2k+1\)), \(G\) is strongly \((2,2(2k+2))\)-bifactorable.
	Let \(\mathbb{F}\) be a strong \((2,2(2k+2))\)-bifactorization of 
\(G\).
	By Theorem~\ref{thm:even-path-decomposition}, \(G\) 
admits an \(\mathbb{F}\)-balanced $(2k+2)$-path decomposition.
\end{proof}

\begin{corollary}
Let \(k\) be a positive integer and let \(r = 
\max\{32(2k+1),(2k+2)(2k+4)\}\). If \(G\) is a bipartite 
\(2(26k + 4r + 22)\)-edge-connected graph such that \(|E(G)|\) is divisible by 
\(2k+2\),
	then \(G\) admits a decomposition into \((2k+2)\)-paths.
\end{corollary}

\begin{proof}
The result follows directly from Theorem~\ref{theorem:conj-equivalence2}.
\end{proof}

\section{Concluding remarks}

This paper benefited greatly from Thomassen's results on decomposition of highly edge-connected graphs.
We hope that the connection of this work to these results of Thomassen  
is clear to a reader familiarized with them.
We also would like to mention that Merker's result~\cite{Me15+} 
contributes to the literature with an alternative to the factorizations results presented in this paper.
%
If one can generalize the Disentangling Lemma to deal with general trees, 
 Merker's result can be applied to solve Conjecture~\ref{conj:dec-bip}.

While writing this paper we learned that 
Bensmail, Harutyunyan, Le, and Thomass{\'e}~\cite{BensmailHLT15+} obtained a result similar to the one presented 
here using a different approach.

The Disentangling Lemma has shown to be a powerful technique to deal with path decompositions.
We were able to use a version of it to decompose regular graphs with prescribed girths into paths of fixed length~\cite{BoMoOsWa15+reg}.
\bibliographystyle{amsplain}
\bibliography{bibliografia-decomposition}
\end{document}